\documentclass[11pt]{amsart}

\usepackage{pdfpages}
\usepackage{amssymb}
\usepackage{amsmath}
\usepackage{amsthm}
\usepackage{thm-restate}
\usepackage{graphicx}
\usepackage{float}
\graphicspath{ {./pics/} }

\newtheorem{theorem}{Theorem}
\newtheorem{corollary}[theorem]{Corollary}
\newtheorem{lemma}[theorem]{Lemma}
\newtheorem{definition}[theorem]{Definition}
\newtheorem{proposition}[theorem]{Proposition}

\newtheorem*{theorem*}{Theorem}
\newtheorem*{remark}{Remark}

\numberwithin{theorem}{section} 
\numberwithin{corollary}{section}
\numberwithin{lemma}{section}
\numberwithin{proposition}{section}
\numberwithin{definition}{section}

\newtheorem{innercustomthm}{Theorem}
\newenvironment{customthm}[1]
	{\renewcommand\theinnercustomthm{#1}\innercustomthm}
	{\endinnercustomthm}

\begin{document}
\title[On The Rigidity Of The Roots Of Power Series]{On The Rigidity Of The Roots Of Power Series With Constrained Coefficients}

\author{Jacob Kewarth}

\begin{abstract}
Here we consider the set $\Sigma_S$ of roots of power series whose coefficients lie in a given set $S$ and how such sets of roots vary as the set $S$ varies. We give an estimate of the depth that complex roots can reach into the disc, offer some criterion for the set of roots to be connected or disconnected, and show that for two finite symmetric sets $S$ and $T$ of integers containing $1$, if $\Sigma_S = \Sigma_T$ then all of their elements between $1$ and $2\sqrt{\max(S)}+1$ must agree.
\end{abstract}

\maketitle


\section{Introduction}

\subsection{History and Outline}

Given a finite (or bounded) set $S$ let $\Sigma_S$ denote the set of roots in $\mathbb{D}$ of power series with coefficients in $S$. Dynamical interest in such sets was started by Barnsley and Harrington (see \textbf{[BHa]}) who studied the family of iterated function systems (IFS) $z \mapsto \lambda z +1$ and $z \mapsto \lambda z - 1$ for various choices of $|\lambda| < 1$. They found that an attractor $A_\lambda$ was connected if and only if $\lambda \in \Sigma_{ \{ 0, \pm 1\} }$, dubbing the set the  ``Mandelbrot set for pairs of linear maps". The set $\Sigma_{ \{ 0, \pm 1 \}}$ was shown to be connected and locally connected by Bousch (see \textbf{[Bo1], [Bo2]}). In \textbf{[Ba]} Bandt proved that $\Sigma_{ \{ 0, \pm 1 \} }$ has at least one (non-trivial) hole and in \textbf{[CKW]} Calegari, Koch, and Walker showed that there are infinitely many holes. \textbf{[SX]} Solomyak and Xu were able to certify a whole `wedge' of $\Sigma_{ \{ 0, \pm 1 \}}$ near the imaginary axis for which polynomial roots are interior points of $\Sigma_{ \{0, \pm 1 \} }$ and \textbf{[CKW]} showed in fact that interior points of $\Sigma_{ \{ 0, \pm 1 \}}$ are dense away from the real axis. In \textbf{[Ti]} Tiozzo related the set $\Sigma_{ \{\pm 1 \} }$ to the set of Galois conjugates of entropies of real quadratic polynomials.

Many variations have since appeared studying the set of roots $\Sigma_S$ for different sets of coefficients $S$. In \textbf{[BY]} the set $\Sigma_S$ was studied for sets $S$ whose elements lie on the unit circle, \textbf{[BBBP]} studied the problem for the set $S = [-g, g]$ with the additional assumption that the power series first coefficient is always 1. \textbf{[Na]} gave a broad criterion for the set of roots of power series starting with $1$ to be connected and locally connected. Motivated by interest in the family of IFS $z \mapsto \lambda z + \xi^j_n$ (where $\xi_n$ is the $nth$ root of unity $e^{2\pi i / n}$) the set of roots of power series starting with $1$ whose coefficients are $\frac{\xi_n^j - \xi_n^k}{1-\xi_n}$ have been studied extensively (see \textbf{[BHu]}, \textbf{[HI]}, and \textbf{[Na]}) and most notably has been shown to be connected, locally connected, and regular-closed for each $n \geqslant 3$.

The object of this paper is to study the rigidity of the set of such roots. We ask how different sets of coefficients yield different sets of roots. Notice that if $S = cT$ for some $c \neq 0$ then $\Sigma_S = \Sigma_T$ and so we further ask if such a criterion is necessary for $\Sigma_S$ to be equal to $\Sigma_T$. Since the roots set is invariant by scaling of the set of coefficients, we mostly consider sets $S$ which are normalized, that is, for which the minimum positive element of $S$ is $1$. Our main result towards this goal is the following:

\begin{customthm}{\ref{weakrigidity}}
Let $S$ and $T$ be finite non-empty symmetric normalized sets of integers. If $\Sigma_S = \Sigma_T$ and $k$ is an integer in $(1, 2\sqrt{\max(S)}+1)$ then $k\in S$ if and only if $k \in T$.
\end{customthm}

Our approach to studying rigidity is to first study how the geometry of $\Sigma_S$ depends on and varies with $S$. The main technical tool will be a series of criteria we develop in \textbf{Section 3} to verify whether an element $z \in \mathbb{D}$ belongs to $\Sigma_S$. 

Pictures of $\Sigma_S$ (where the elements of $S$ do not have big gaps) tend to display two prominent features. Firstly, the real line contains an interval which seems to extend far deeper into the unit disc than any complex roots can reach, and secondly there appears to be an annulus contained in $\Sigma_S$. Using the criteria of $\textbf{Section 3}$ we certify both claims for appropriate sets $S$. In particular, if we consider:

\begin{definition}
If $S$ is a non-empty normalized finite set of real numbers we define:
$$\rho_{inn}(S) = \inf \{R > 0 \mid A_R \subseteq \Sigma_S \} $$
where $A_R$ is the annulus $\{ R \leqslant |z| < 1\}$.
\end{definition}

When discussing coefficient sets which $S$ which are normalized it helps to consider how far apart consecutive elements of $S$ are, we call the total gap of $S$ the maximum such distance. Then when the total gap of $S$ is small enough we prove that the subset $\Sigma_S^1$ of $\Sigma_S$ of roots of power series with coefficients in $S$ which start with $1$ contains an annulus giving a upper bound on $\rho_{inn}$.

\begin{customthm}{\ref{annulus}}
Let $S$ be a symmetric non-empty finite set of real numbers with total gap $1$ and let $M = \max_{s \in S} |s|$. Then $\Sigma^1_S$ contains the annulus $\frac{1}{\sqrt{M}} \leqslant |z| < 1$. In particular, we have that:
$$\rho_{inn}(S) \leqslant \frac{1}{\sqrt{M}} $$
\end{customthm}

In $\textbf{Section 4}$ we consider the largest possible annulus containing the complex roots of $\Sigma_S$. The reason for avoiding real roots is that real roots of $\Sigma_S$ tend to get far deeper than complex roots. Indeed, real roots may get as deep at $\frac{1}{M+1}$ whereas it is well known that complex roots may not be smaller than $\frac{1}{\sqrt{M}+1}$. This causes the real roots to stick out giving the appearance of an `antenna'.

\begin{definition}
If $S$ is a non-empty normalized finite set of real numbers we define:
$$ \rho_{out} = \sup \{ R > 0 \mid (\Sigma_S \setminus \mathbb{R}) \subseteq A_R \} $$
where $A_R$ is the annulus $\{ R \leqslant |z| < 1\}$.
\end{definition}

\begin{remark}
Note that by definition we have $\rho_{out} \leqslant \rho_{inn}$.
\end{remark}

To approach the problem, we first develop a general inequality on the size of all roots of a power series which may be of independent interest:

\begin{customthm}{\ref{prod ineq}}
If $S$ is a (possibly infinite) bounded non-empty set of (possibly complex) numbers and $P(z) = \sum_{j=0}^\infty p_jz^j$ is a power series with coefficients in $S$ and with $p_0 =1$ then for any roots $\lambda_1, ..., \lambda_k \in \mathbb{D}$ of $P(z)$, we have that:

$$\prod_{j=1}^k (|\lambda_j|^{-1} - 1) \leqslant \sup_j |p_j|.$$

In particular, $P(z)$ does not have $k$ roots $\lambda_1, ..., \lambda_k$ all with $|\lambda_j| < \frac{1}{\sup_j|p_j|^{1/k} + 1}$.

\end{customthm}

Using this inequality we given an alternative proof of the well known estimate that $\rho_{out} \geqslant \frac{1}{\sqrt{M}+1}$. The remainder of the first part of section 4 is dedicated to showing this estimate is never best possible. Namely, we will prove the following:

\begin{customthm}{\ref{rho-out}}
Let $S$ be a finite non-empty symmetric normalized set of real numbers and let $M = \max(S)$. Then there is no sequence of numbers $\lambda_n \in \Sigma_S \setminus \mathbb{R}$ such that $|\lambda_n| \rightarrow \frac{1}{\sqrt{M}+1}$. Moreover, we have that:

$$\frac{1}{\sqrt{M}+1} < \rho_{out}(S) \leqslant \frac{1}{\sqrt{M+k}} $$
for $k = \max \{ s \in S \mid 1 \leqslant s < 2\sqrt{M} + 1 \}$.
\end{customthm}

The second part of section 4 gives a brief estimates on how deep roots can get at each angle $\theta$.

In \textbf{Section 5} we use the work of \textbf{[Na]} to give the following classification for connectedness:

\begin{customthm}{\ref{conn-crit}}
If $S$ is a normalized symmetric set of integers containing $0$ then:
	\begin{enumerate}
		\item If the total gap of $S$ is $3$ or more then $\Sigma_S$ is disconnected.
		\item If the total gap of $S$ is $1$ or $2$ and if $M = \max(S) \geqslant 40$ then $\Sigma_S$ is connected and locally connected.
	\end{enumerate}
\end{customthm}

Finally, \textbf{Section 6} explores the question of rigidity offering a proof of \textbf{Theorem \ref{weakrigidity}} along with a geometric interpretation of the result. Additionally we show a weaker sort of quasi-rigidity for the larger coefficients:

\begin{customthm}{\ref{quasirigidity}}
If $S$ and $T$ are finite non-empty normalized sets of integers and if $\Sigma_S = \Sigma_T$ then $k \in S$ implies that either $k$, $k-1$, or $k+1$ is in $T$.
\end{customthm}

\section{Definitions and Basic Properties}

We begin by introducing some basic definitions and review well known elementary properties about sets of roots.

\begin{definition}
Let $S$ be a non-empty bounded set of complex numbers. We define the sets:

\[
\Sigma_S = \left \{ z \in \mathbb{D} \mid \exists p_0, p_1, ... \in S, p_0 \neq 0, \sum_{j=0}^\infty p_jz^j = 0 \right \} 
\]

\[
\Sigma_S^1 = \left \{ z \in \mathbb{D} \mid \exists p_0, p_1, ... \in S, p_0 = 1, \sum_{j=0}^\infty p_jz^j = 0 \right \} 
\]

\[
\Sigma_S^{(fin)} = \left \{ z \in \mathbb{D} \mid \exists p_0, p_1, ..., p_n \in S, p_0 \neq 0, \sum_{j=0}^n p_jz^j = 0 \right \}
\]

\end{definition}

The above definition allows for the case that $S$ is infinite, however in most applications we will take $S$ to be a finite set. Boundedness of $S$ guarantees that even if it is infinite the radius of convergence of any power series in $S$ will be at least $1$ and so we avoid convergence issues by restricting to the unit disc. Unless explicitly stated otherwise we will always assume that $S$ is finite.

For a given polynomial $f(z)$ of degree $n$ with coefficients in $S$, we can consider the power series $P(z) = \sum_{j=0}^\infty f(z)z^{jn}$. This power series shares the same roots as $f(z)$ inside the unit disc and hence $\Sigma_S$ will contains the roots of all polynomials with coefficients in $S$ which lie in $\mathbb{D}$. In fact, we can equivalently view $\Sigma_S$ as the closure of the set of roots of polynomials with coefficients in $S$ restricted to the unit disc.

A well known estimate about roots of polynomials says that if $P(z) = p_0 + p_1z + ... + p_nz^n$ is a polynomial then the roots of $P$ all lie in the annulus $\frac{1}{M+1} \leq |z| \leq M+1$ for $M = \max \{\frac{|p_j|}{|p_i|} \mid 0 \leqslant i, j \leqslant n, p_i \neq 0$\}. 

Since the roots of power series with coefficients in $S$ are limits of the roots of polynomials with coefficients in $S$ we get the following lemma:

\begin{lemma}
	\label{outerannulus}
	 If $S$ is a non-empty finite set of complex numbers other than $\{ 0 \}$ then:
	\begin{enumerate}
		\item $\Sigma_S$ is contained in the annulus $\frac{1}{M+1} \leq |z| < 1$ for $M = \max_{t \neq 0, s, t \in S} \frac{|s|}{|t|}$. 

		\item If $S$ contains only real numbers and the signs of the minimum and maximum non-zero elements of $S$ (in absolute value) have different signs then the bound is achieved by a multiple of the power series $1 - Mz - Mz^2 - Mz^3 - ...$.
	\end{enumerate}
\end{lemma}

\begin{definition}
We say that $S$ is symmetric if $S = -S$.
\end{definition}

If $S$ is symmetric then $\Sigma_S$ has a symmetry that $z \in \Sigma_S$ if and only if $-z \in \Sigma_S$. Further, if $S$ contains only real numbers then $z \in \Sigma_S$ if and only if $\overline{z} \in \Sigma_S$. Thus for symmetric sets of real numbers one can study $\Sigma_S$ only in the first quadrant of the plane.

In proofs we will often consider consecutive elements $s, t$ of $S$ and how far apart they are. Since $\Sigma_S = \Sigma_{cS}$ for any constant $c \neq 0$ it often proves helpful to assume the smallest (non-zero positive) element of $S$ is $1$. 

\begin{definition}
A finite non-empty set of real numbers $S$ is said to be normalized if $1 \in S$ and for every $s \in S$, $s \neq 0$, we have $|s| \geqslant 1$. If $S$ is a normalized set and $T = cS$ for some non-zero constant $c$ we call $S$ a normalization of $T$.
\end{definition}

Any non-empty finite set of real numbers other than $\{ 0 \}$ can be normalized by dividing by one of $\pm \min\{|s| \mid s \in S, s \neq 0 \}$, however such a normalization is only unique up to a possible $\pm$. For example $\{-2, 2, 4 \}$ could be normalized as $\{-1, 1, 2 \}$ or as $\{1, -1, -2 \}$. Whenever $S$ is symmetric it follows that $S$ has a unique normalization. 

\begin{definition}
Let $S$ be a non-empty finite set of real numbers. Let $m = \min \{|s| \mid s \in S, s \neq 0 \}$. Given $s, t \in S$ we say that the relative gap between $s$ and $t$ is $\frac{|s-t|}{m}$. In particular if $S$ is normalized then the relative gap is just the difference between $s$ and $t$. If $S = \{s_1, s_2, ..., s_n \}$ with $s_j < s_{j+1}$ then we say that the total gap of $S$ is the $\max_j \frac{|s_j - s_{j+1}|}{m}$.
\end{definition}

\section{Recursive Criterions}

Here we prove various recursive criterion to certify if an element is in $\Sigma_S$ or not and then use them to verify that when $S$ does not have large gaps then $\Sigma_S$ contains both an annulus and two intervals of the real line. The criterion explicitly construct power series with desired roots (or prove no such power series exists). The presentation here is purely analytic but the same recursion as in $\textbf{Theorem \ref{1-step}}$ can be found in \textbf{[So]} and \textbf{[Ba]} with a more dynamical interpretation. On the other hand, the recursion in \textbf{Theorem \ref{2-step}} is new.

\begin{theorem}(1-step recursion)
\label{1-step}
Let $S$ be a non-empty bounded (possibly infinite) set of complex numbers. Given $r \in \mathbb{C}$ with $|r| > 1$, the point $1/r$ belongs to $\Sigma_S$ if and only if there exists a sequence $(p_j)_{j=0}^\infty$ in $S$ with $p_0 \neq 0$ such that the sequence $(q_j)_{j=0}^\infty$ defined by:
\begin{align*}
q_0 = p_0, & & q_{j+1} = p_{j+1} + rq_j
 \end{align*}
 is uniformly bounded.
\end{theorem}

\begin{proof}
First suppose we have a sequence $(p_j)_{j=0}^\infty$ such that the sequence $(q_j)_{j=0}^\infty$ is uniformly bounded. Consider the induced power series $P(z) = \sum_{j=0}^\infty p_jz^j$ and the quotient $Q(z) = \frac{P(z)}{1-rz}$. Note that since the $p_j$ are uniformly bounded, the radius of convergence of $P(z)$ is at least 1 and hence holomorphic on the unit disc. Thus $P(1/r) = 0$ if and only if $Q(z)$ is holomorphic on the unit disc. Even if $P(1/r) \neq 0$ the only pole $Q(z)$ could have in $\mathbb{D}$ will be at $1/r$ and so $Q(z)$ is at least holomorphic in a neighbourhood of $0$. Thus it has a power series expansion $Q(z) = \sum_{j=0}^\infty q_jz^j$ for $|z| < 1/r$. Then we have $P(z) = (1-rz)Q(z)$ and so the coefficients of $\sum_{j=0}^\infty q_jz^j$ satisfy the recursion $p_0 = q_0$ and $p_{j+1} = q_{j+1} - rq_{j}$. Since the $q_j$ are uniformly bounded it follows the radius of convergence of $\sum_{j=0}^\infty q_jz^j$ is at least $1$ and so $Q(z)$ is holomorphic on the entire unit disc implying $P(1/r) = 0$.

\vspace{10pt}

Suppose now that the sequence $(q_j)_{j=0}^\infty$ is unbounded. The construction before still holds and so we have $Q(z) = \frac{P(z)}{1-rz}$ has power series expansion $\sum_{j=0}^\infty q_jz^j$ which satisfies the previously mentioned recursion. Let $M = \sup_{s \in S} |s|$ and let $|r| = 1+\epsilon$. Since $|r| > 1$ we have $\epsilon > 0$. As $q_j$ is unbounded we can pick some index $j$ such that $|q_j| \geqslant \frac{M}{\epsilon} + \eta$ for some $\eta > 0$. 

We now claim for every $k \geqslant 0$ that:
\begin{align*}
|q_{j+k+1}| &\geqslant |g_j| + |r|^k\epsilon\eta.
\end{align*}

For the base case suppose $k=0$. Since $q_{j+1} = p_{j+1} + rq_j$ we have that:
\begin{align*}
|q_{j+1}| &\geqslant |r| \left( \frac{M}{\epsilon} + \eta \right) - |p_{j+1}| \\
&= (1+\epsilon) \left(\frac{M}{\epsilon} + \eta \right) - |p_{j+1}| \\
&= \frac{M}{\epsilon} + \eta + M + \epsilon\eta - |p_{j+1}| \\
&\geqslant |q_j| + M + \epsilon\eta - M \\
&= |q_j| + \epsilon \eta.
\end{align*}

Now suppose that for some $k$ we have $|q_{j+k+1}| \geqslant |g_j| + |r|^k\epsilon\eta$. Then we have that:
\begin{align*}
|q_{j+k+2}| &\geqslant |r||q_{j+k+1}| - |p_{j+k+1}| \\
&\geqslant |r|(|q_j| +|r|^{k}\epsilon\eta) - M \\
&\geqslant |q_j| + |r|^{k+1}\epsilon\eta.
\end{align*}

Thus the $q_j$ grow (at least) exponentially. This implies that the radius of convergence of $\sum_{j=0}^\infty q_jz^j$ is strictly less than $1$ and so $Q(z)$ is not holomorphic on the whole unit disc which implies $P(1/r) \neq 0$ proving the claim.

\end{proof}

An important consequence of the above theorem is that it gives an explicit bound (dependent on the root and on $S$) on how large the $q_j$ can get before we know they tend to $\infty$. In particular, if the size of $r$ is close to $M$ then $\frac{M}{|r|-1}$ will be close to $1$ and so showing that the $q_j$ become slightly bigger than $1$ is enough to prove that $1/r$ is not a root of a power series.

\begin{definition}
Let $S$ be a non-empty bounded normalized set such that $\min_{s \in S, s \neq 0} |s| = 1$ and let $M = \sup_{s \in S}|s|$. For $\lambda = \frac{1}{r} \in \mathbb{D}$ we define the (1-step) escape threshold of $\lambda$ with respect to $S$ as:
\begin{equation*}
E_1 = E_1(\lambda, S) = \frac{M}{|r| - 1}
\end{equation*}
\end{definition}

\begin{corollary}
Let $S$ be a non-empty bounded normalized set and let $M = \sup_{s \in S} |s|$. Let $(p_j)_{j=0}^\infty$ be some sequence in $S$ and let $\lambda \in \mathbb{D}$. Let $q_j$ be as in the 1-step recursion. If there is some index $j$ such that $|q_j| > E_1(S, \lambda)$ then $\lambda$ is not a root of $\sum_{j=0}^\infty p_jz^j$.
\end{corollary}

Also note that we need $|q_j|$ to be strictly larger than the escape threshold to conclude $1/r$ is not a root of $P(z)$. If $|q_j|$ is exactly equal to the escape threshold then tracing through the proof shows that in fact $|q_{j+1}|$ is also equal to the escape threshold and that $|p_{j+1}| = M$ and hence that for all indexes $k$ after $j$, we will have $|q_{k}|$ equal to the escape threshold and $|p_k|$ equal to $M$. This observation provides a useful tool to see how many different power series have a specified root. If you can prove that for some root and some choices of coefficients $p_0, p_1, ..., p_j$ that $|q_j|$ is equal to the escape threshold it follows every coefficient of the power series after index $j$ is forced. In fact, $|q_j|$ being merely `close' to the escape threshold can sometimes be enough to force the values of $p_{j+k}$.

Lastly, the escape threshold above is best possible since we can find power series with roots such that $|q_j|$ achieves the escape threshold. For example, if $P(z) = 1-z-z^2-z^3-....$ then $M = 1$ and $P(z)$ has a root at $1/2$. Thus the escape threshold is $\frac{1}{2-1} = 1$. The recursion says $q_0 = p_0 = 1$ and so at the first step we have achieved the escape threshold.

This recursion also offers an easy algorithm to generate images of the set $\Sigma_S$ called Bandt's algorithm. Indeed, for any point $\lambda$ we can decide if $\lambda$ is in $\Sigma_S$ or not by iterating the recursion over every possible choice of $p_j \in S$ to get a $q_j$. If the $q_j$ exceeds the escape threshold then $\lambda$ is not a root of that particular choice and we stop considering it, if every possible choice of $p_j$ exceeds the escape threshold then $\lambda$ is not in $\Sigma_S$. If after some predetermined number of iterations we find all the $q_j's$ are smaller than the escape threshold we assume $\lambda$ is in $\Sigma_S$. Note also that since $\Sigma_S$ is symmetric across both the real and imaginary axis whenever $S$ is a symmetric set of real numbers we can just generate the image of $\Sigma_S$ in the first quadrant.

\begin{figure}[!h]
\centering
\includegraphics{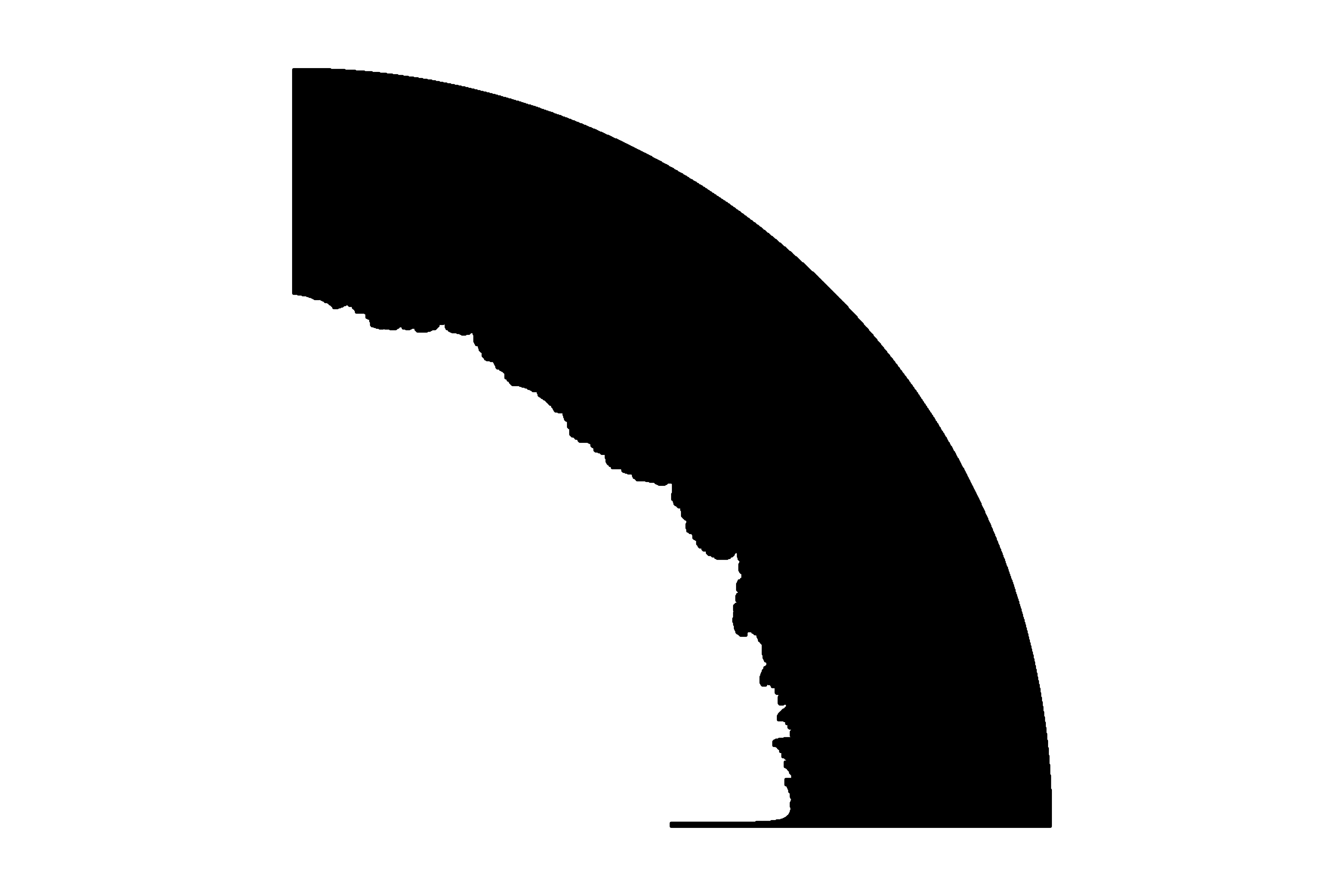}
\caption{$\Sigma_{\{0, \pm 1 \} }$ in the first quadrant}
\end{figure}

\begin{theorem}
\label{real-strip}
Suppose $S$ is a non-empty normalized finite set of real numbers. Denote $M = \max_{s \in S} |s|$ and assume $M, -M \in S$. If the total gap of $S$ is $\leqslant 2$ then $\Sigma_S$ contains the real interval $[\frac{1}{M+1}, 1)$.
\end{theorem}

\begin{proof}
Fix $1/r$ such that $1 < r \leqslant M+1$. Let $p_0 = 1 \in S$. Then $q_0 = 1$. Suppose we have defined the first $j$ terms in our sequence $p_0, p_1, ..., p_j$ such that the corresponding $q_0, q_1, ..., q_j$ are all $\leqslant 1$ in absolute value. Then we have that $0 \leqslant |rq_j| \leqslant M+1$. If $|rq_j| \geqslant M$ then we can pick $p_{j+1}$ to be $\pm M$ and have that $|q_{j+1}| \leqslant 1$. Otherwise we have that $0 \leqslant |rq_j| \leqslant M$. Thus since the total gap of $S$ is $\leqslant 2$, $S$ must contain an element in $[-rq_j - 1, -rq_j+1]$ and so we can pick $p_{j+1}$ to be such an element and we have that $q_{j+1} = p_{j+1} + rq_j$ is in $[-1, 1]$.

Thus we have a sequence $(p_j)_{j=0}^\infty$ in $S$ such that the sequence $(q_j)_{j=0}^\infty$ is uniformly bounded by $1$ and so $1/r \in \Sigma_S$ as wanted.
\end{proof}

\begin{corollary}
Suppose that $S$ is a non-empty symmetric normalized finite set of integers and that for every integer $k$ satisfying $|k| < \max(S)$  at least one of $k$ and $k+1$ is in $S$. Then $\Sigma_S$ contains the real interval $[\frac{1}{M+1}, 1)$.
\end{corollary}

\begin{remark}
A consequence of the above theorem is that we can get the best possible real interval if $S$ contains only every other integer between $1$ and our maximum. This implies that it is possible for $\Sigma_S \cap \mathbb{R} = \Sigma_T \cap \mathbb{R}$ even when $S$ and $T$ are wildly different. For example, let $S = \{ \pm 1, \pm 3 \}$ and let $T = \{ \pm 1, \pm 2, \pm 3 \}$. Then by the above we know that $\Sigma_S \cap \mathbb{R} = \Sigma_T \cap \mathbb{R}$. On the other hand, by \textbf{Theorem \ref{weakrigidity}} below, since $2 < 2\sqrt{3}+1$ it follows that $\Sigma_S \neq \Sigma_T$.
\end{remark}

\begin{figure}[!h]
\centering
\includegraphics{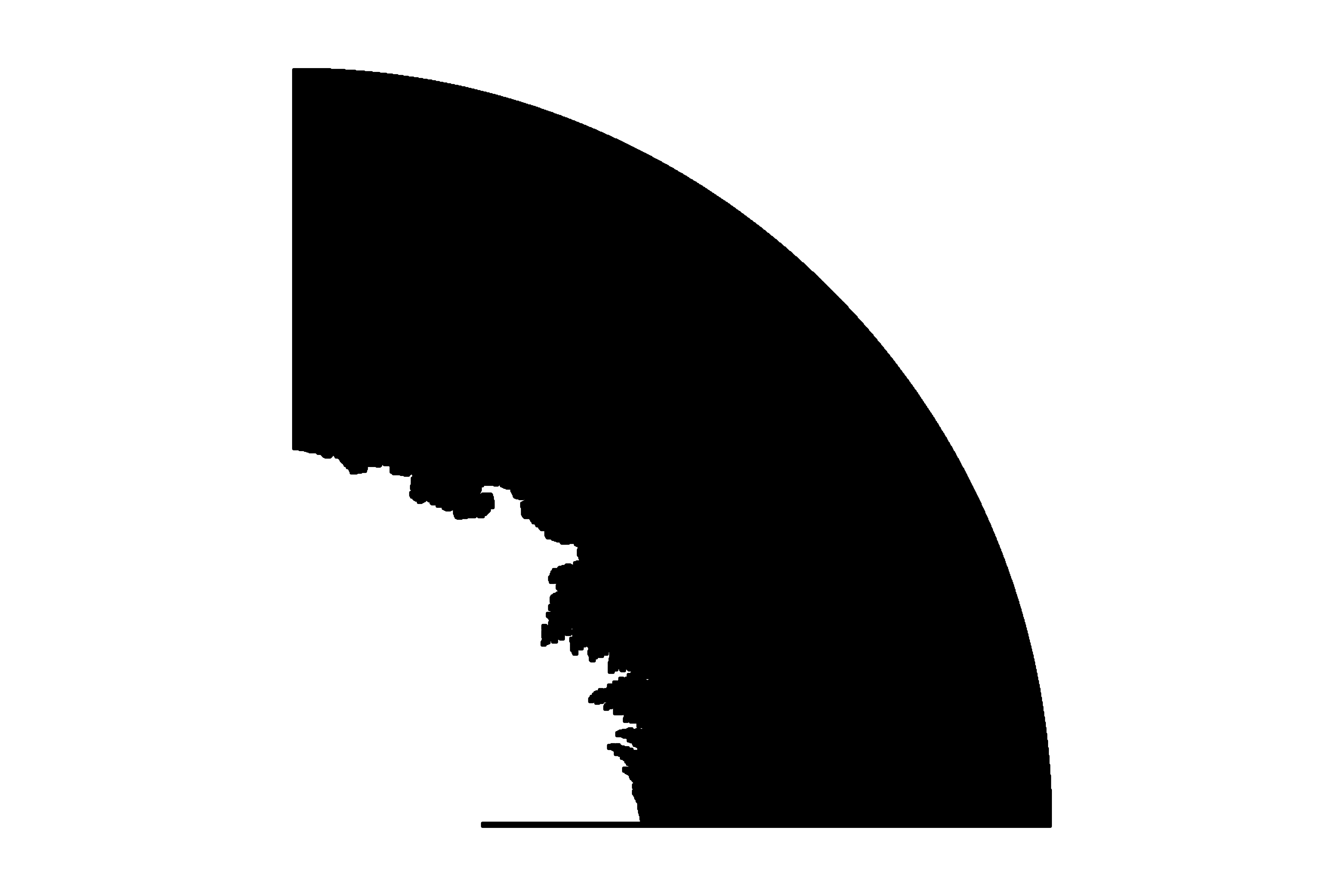}
\caption{$\Sigma_{\{0, \pm 1 , \pm 3\} }$ in the first quadrant. Even though the coefficients skip 2 we have $\Sigma_{\{0, \pm 1 , \pm 3\} } \cap [0, \infty)$ is connected.}
\end{figure}

\subsection{Complex Roots}

The previous argument worked in part because the root we looked at and our set $S$ were both real. If we took the potential root to be complex (while keeping $S$ real) then the recursion is difficult to control since each $p_j$ is real while the $q_j$ quickly become imaginary, giving a system of recursive equations for the real and imaginary parts. As $p_j$ will be real we can easily control the size of the real part of the recursion however it is difficult to stop the imaginary part from growing. Because of this the recursion is impractical to solve. Instead we leverage the fact that the complex roots of polynomials and power series with real coefficients must come in pairs.

\begin{theorem}(2-Step Recursion)
\label{2-step}
Suppose $S$ is a non-empty bounded (possibly infinite) set of real numbers and that $\lambda = \frac{1}{r}e^{i\theta} \in \mathbb{D} \setminus \mathbb{R}$. Then $\lambda \in \Sigma_S$ if and only if there exists a sequence $(p_j)_{j=0}^\infty$ in $S$ with $p_0 \neq 0$ such that the sequence $(q_j)_{j=0}^\infty$ defined by:
\begin{equation} \label{rec2}
	q_0 = p_0, \qquad q_1 = p_1 + 2r\cos(\theta) q_0, \qquad q_{j+2} = p_{j+2} + 2r\cos(\theta)q_{j+1} - r^2q_j
\end{equation}
is uniformly bounded.
\end{theorem}

\begin{proof}
First note that since $S$ contains only real numbers we know that $\lambda$ and $\overline{\lambda}$ either are both roots of the same power series/polynomial with coefficients in $S$ or they are both not roots. Suppose first that $\lambda$ is the root of some polynomial/power series $P(z)$. Then it follows that the power series $Q_1(z) = \frac{P(z)}{1-z/\lambda}$ has uniformly bounded coefficients and that $\overline{\lambda}$ is a root of $Q_1(z)$. Thus we have that the power series $Q(z) = \frac{Q_1(z)}{1-z/\overline{\lambda}} = \frac{P(z)}{1-2r\cos(\theta)z + r^2z^2}$ has uniformly bounded coefficients also. The coefficients of $Q(z)$ are as in \eqref{rec2} proving the implication.

\vspace{10pt}

For the other direction we argue by converse. If $\lambda$ and $\overline{\lambda}$ are not roots of a power series $P$ then the power series $Q(z) = \frac{P(z)}{1-2r\cos(\theta)z + r^2z^2}$ has a pole at $\lambda$. This means the radius of convergence is $< 1$ and so the coefficients must be unbounded.

\end{proof}

Notice that the above recursion relies not just on the size of the root $|r|$ but also on $|r|^2$ and on the argument of the root. This is highly suggestive that the complex roots are necessarily smaller and that some angles may have roots deeper into the disc than others.

To quantify this, recall that for the 1-step recursion, if $|q_j|$ ever exceeds the threshold $\frac{M}{|r|-1}$ then we can say that $q_j$ is unbounded. Since we apply this theorem twice in our 2-step recursion it follows that if for $1/r \in \Sigma_S$ that if $M$ is the supremum value of the coefficients of $P(z)$ then $\frac{M}{|r|-1}$ is the maximum possible value for the coefficients of $\frac{P(z)}{1-rz}$ and hence that $\frac{M}{(|r|-1)^2}$ is the maximum possible value for $\frac{P(z)}{(1-rz)(1-\overline{r}z)}$.

We know that the smallest coefficient we can initially choose is $1$ and so if this new escape threshold is ever $< 1$ then $1/r$ can not be in $\Sigma_S$. Setting $\frac{M}{\epsilon^2} < 1$ and rearranging gives that $\epsilon > \sqrt{M}$ and hence $|r| > \sqrt{M}+1$. Thus $\Sigma_S \setminus \mathbb{R}$ is contained in the annulus $(\sqrt{M}+1)^{-1} \leqslant |z| \leqslant \sqrt{M}+1$. Later we will generalize this analysis to get a bound on the product of an arbitrary number of roots.

\begin{definition}
Let $S$ be a non-empty bounded normalized set of real numbers such that $1$ is the minimal positive element of $S$ and let $M = \sup(S)$. For $\lambda = \frac{1}{r} \in \mathbb{D} \setminus \mathbb{R}$ we define the (2-step) escape threshold of $\lambda$ with respect to $S$ as:
\begin{equation*}
E_2 = E_2(\lambda, S) = \frac{M}{(|r| - 1)^2}
\end{equation*}
\end{definition}

\begin{corollary}
Let $S$ be a non-empty bounded normalized set of real numbers and let $M = \sup_{s \in S}|s|$. Let $(p_j)_{j=0}^\infty$ be some sequence in $S$ and let $\lambda \in \mathbb{D} \setminus \mathbb{R}$. Let $q_j$ be as in the 2-step recursion. If there is some index $j$ such that $|q_j| > E_2(S, \lambda)$ then $\lambda$ is not a root of $\sum_{j=0}^\infty p_jz^j$.
\end{corollary}

As with the 1-step recursion we can use the 2-step to get an algorithm for generating $\Sigma_S$. This is worse in that the 1-step works for any set $S$ whereas the 2-step requires $S$ to contain only real numbers, moreover the 2-step does not generate pictures of real roots. On the other hand it avoids any computations involving complex numbers and therefore may sometimes be preferable.

Finally we use the 2-step recursion to prove that when the total gap of $S$ is small then $\Sigma_S$ contains an annulus.

\begin{theorem}
\label{annulus}
Let $S$ be a symmetric non-empty finite set of real numbers with total gap $1$ and let $M = \max_{s \in S} |s|$. Then $\Sigma^1_S$ contains the annulus $\frac{1}{\sqrt{M}} \leqslant |z| < 1$. In particular, we have that:
$$\rho_{inn}(S) \leqslant \frac{1}{\sqrt{M}} $$
\end{theorem}

\begin{proof}
Without loss of generality assume $S$ is normalized. Fix a complex number $z = re^{i\theta}$ with $1 \leqslant r \leqslant \sqrt{M}$. If $z$ is real then the claim follows from the previous theorem so without loss of generality we can assume $z$ is an imaginary number. Moreover, by the symmetries of $\Sigma_S$ we can assume $\cos(\theta) \geqslant 0$. We will argue by the 2-step recursion.

Pick $p_0 = 1$. Then it follows that:
\begin{align*}
p_1 &= q_1 - 2r \cos(\theta).
\end{align*}

We can estimate $0 \leqslant 2r\cos(\theta) \leqslant 2\sqrt{M} \leqslant M+1$.
Since the total gap is $\leqslant 1$ we can pick $p_1$ to be an element in $[-M, M]$ such that $0 \leqslant q_1 \leqslant 1$. Suppose we have defined $p_0$, $p_1$, ..., $p_j$ and that the corresponding $q_0, q_1$, ..., $q_j$ satisfy $0 \leqslant q_k \leqslant 1$.
Now consider the recursion:
\begin{align*}
p_{j+1} &= q_{j+1} - 2r\cos(\theta)q_{j} + r^2q_{j-1}.
\end{align*}
 
 Or put another way:
 \begin{align*}
 q_{j+1} &= p_{j+1} + 2r \cos(\theta)q_j - r^2 q_{j-1}.
 \end{align*}
 
 We can estimate the size of the second and third term by:
 \begin{align*}
 -M \leqslant 2r\cos(\theta)q_j - r^2q_{j-1} \leqslant 2\sqrt{M} \leqslant M+1.
 \end{align*}

Again since the total gap is $\leqslant 1$ we may pick $p_{j+1}$ so that $0 \leqslant q_{j+1} \leqslant 1$

Since the sequence $q_j$ that we have constructed is uniformly bounded it follows that $1/r \in \Sigma_S$ as wanted.
\end{proof}

At this point it is natural to ask if this annulus is the largest contained in $\Sigma_S$ for $S$ a finite non-empty symmetric set of real numbers with total gap 1. A quick computation shows that for the angle $\pi/2$ that $\Sigma_S$ contains at most $i[\frac{1}{\sqrt{M+1}}, 1)$. Thus the best possible annulus would be only slightly more optimal. Since $\Sigma_{0, \pm 1}$ is known to contain an annulus $\frac{1}{\sqrt{2}} \leqslant |z| < 1$ it seems plausible that the best possible annulus is in fact $\frac{1}{\sqrt{M+1}} \leqslant |z| < 1$.

\subsection{Polynomials}
Since a polynomial can be thought of as a power series whose terms are eventually $0$ it follows that we can specialize the previous 1-step recursion to get the following:

\begin{theorem}
If $P(z) = p_0 + p_1z + ... + p_nz^n$ is a polynomial and $1/r \in \mathbb{D}$ then $P(1/r) = 0$ if and only if the numbers $q_0, q_1, ..., q_{n}$ defined by:
\begin{align*}
q_0 = p_0, & & q_{j+1} = p_{j+1} + rq_j
\end{align*}
satisfies $q_n = 0$.
\end{theorem}

In fact we could generalize this to an arbitrary number of roots and say that for $P(z)$ to have roots $1/r_1, ..., 1/r_k$ we must have that the coefficients $q_j$ of the power series expansion of $Q(z) = \frac{P(z)}{(1-r_1z)...(1-r_kz)}$ at $0$ satisfies $q_n = q_{n-1} = ... = q_{n-k+1} = 0$.

\section{How Deep Can The Roots Get?}

\subsection{The Depth of Complex Roots}

As the title says, the goal of this section is to discuss how deep complex roots can actually get into the unit disc. Recall that for $S$ a non-empty set of real numbers, we define:
\begin{align*}
\rho_{out}(S) = \sup \left\{ R > 0 \mid (\Sigma_S\setminus \mathbb{R}) \subseteq A_R \right\}
\end{align*}
where $A_R$ is the annulus $\{ R \leqslant |z| < 1 \}$.

We begin with a broad result for how many roots can simultaneously be deep in the disc and then turn specifically to complex roots.

\begin{theorem}
\label{prod ineq}
If $S$ is a (possibly infinite) bounded non-empty set of (possibly complex) numbers and $P(z) = \sum_{j=0}^\infty p_jz^j$ is a power series with coefficients in $S$ and with $p_0 =1$ then for any roots $\lambda_1, ..., \lambda_k \in \mathbb{D}$ of $P(z)$, we have that:

$$\prod_{j=1}^k (|\lambda_j|^{-1} - 1) \leqslant \sup_j |p_j|.$$

In particular, $P(z)$ does not have $k$ roots $\lambda_1, ..., \lambda_k$ all with $|\lambda_j| < \frac{1}{\sup_j|p_j|^{1/k} + 1}$.

\end{theorem}

\begin{proof}
Suppose there is such a power series $P(z) = \sum_{j=0}^\infty p_jz^j$ with roots $\lambda_1 = \frac{1}{r_1}e^{i\theta_1}, ..., \lambda_k = \frac{1}{r_k}e^{i\theta_k}$ such that $(r_1-1)(r_2-1)...(r_k-1) > M = \sup_j |p_j|$.

For each index $j$ we define $Q_j(z) = \frac{P(z)}{(1-z/\lambda_1)...(1-z/\lambda_j)}$. Notice that $Q_j(z) = \frac{Q_{j-1}(z)}{1-z/\lambda_j}$ for $1 < j \leqslant k$.

Since each $Q_j(z)$ is holomorphic in some neighbourhood of $0$ we can write $Q_j(z) = \sum_{n=0}^\infty q_n^j z^n$. We know from the 1-step recursion proof that the escape threshold for $Q_1(z)$ is $\frac{M}{(r_1-1)}$ and hence the coefficients of the power series expansion of $Q_1(z)$ are at most$\frac{M}{(r_1-1)}$. Repeating it follows that the escape threshold for each $Q_j(z)$ is at most $\frac{M}{(r_1-1)...(r_j-1)}$ and hence that the coefficients of the power series expansion of $Q_j(z)$ are at most $\frac{M}{(r_1-1)...(r_j-1)}$.

In particular, we have that the escape threshold for $Q_k(z)$ is $\leqslant \frac{M}{(r_1-1)...(r_k-1)} < \frac{M}{M} = 1$ and hence that the coefficients of $Q_k(z)$ are all $<1$. But since $p_0 = 1$ it follows that the power series expansion of $Q_k(z)$ also starts with $1$ which gives a contradiction.
\end{proof}

\begin{remark}
The above theorem essentially says that for any polynomial $P$, the product of all of its roots (minus $1$) bigger than $2$ (times $|p_n|$) is bounded above by the largest coefficient of $P$, that is, by its $\infty-norm$. This closely resembles Landau's inequality \textbf{[La]}:
$$|p_n| \prod_{j=1}^n\max(1, |\lambda_j|) \leqslant ||P||_2 = \sqrt{\sum_{j=1}^n |p_j|^2}. $$
\end{remark}

Let $S$ be a finite symmetric normalized set of real numbers and let $M = \max(S)$. As a consequence of the above, we have that $\Sigma_S$ is contained in the annulus $\frac{1}{M+1}$. This is best possible since we can reach a root of size $\frac{1}{M+1}$ by looking at $1 - \sum_{j=1}^\infty Mz^j$. On the other hand this root is real and so it raises the question if a complex root could get this deep. Since a power series with real coefficients has a complex root $\lambda$ if and only if it also has $\overline{\lambda}$ as a root, it follows from the above theorem that:

\begin{corollary}
If $S$ is a symmetric non-empty finite set of real numbers and $M = \max(S)$ then $\rho_{out}(S) \geqslant \frac{1}{\sqrt{M}+1}$.
\end{corollary}

Note that this estimate is well known. See for example $\bf{[BBBP]}$. The rest of this section is dedicated to showing that this annulus is not best possible.

\begin{lemma}
Let $S$ be a finite non-empty symmetric normalized set of real numbers and let $M = \max(S)$. Then:
\begin{enumerate}
\item If $P(z) = \sum_{j=0}^\infty p_jz^j$ is a power series with $p_j \in S$ and $p_0 \neq 0$ has two roots $\lambda_1, \lambda_2$ with $|\lambda_j| = \frac{1}{\sqrt{M}+1}$ and such that $\lambda_1 = \overline{\lambda_2}$ and $\lambda_1$ is in the first quadrant. Then $\lambda_1 = \lambda_2 = \frac{1}{\sqrt{M}+1}$ and (after possibly multiplying by -1) $P(z) = 1 - (2\sqrt{M}+1)z + \sum_{j=2}^\infty Mz^j$.

\item If $2\sqrt{M} + 1 \in S$ and if $P_n(z) = \sum_{j=0}^\infty p_{n, j}z^j$ are power series with roots $\lambda_{n, 1}$ and $\lambda_{n, 2}$ with $|\lambda_{n, j}| \rightarrow \frac{1}{\sqrt{M}+1}$ and $p_{n, j} \in S$, $p_{n, 0} = 1$, and $p_{n, 1} = -(2\sqrt{M}+1)$ then for $n$ large enough $\lambda_{n, j} \in \mathbb{R}$ and for every $m$ there is an $N$ such that $n > N$ implies that $p_{n, k} = p_k$ for all $k = 0, 1, ..., m$.

\item if $P_n(z) = \sum_{j=0}^\infty p_{n, j}z^j$ are power series with roots $\lambda_{n, 1}$ and $\lambda_{n, 2}$ with $|\lambda_{n, j}| \rightarrow \frac{1}{\sqrt{M}+1}$ and $p_{n, j} \in S$ then for $n$ large enough (and after possibly multiplying by -1) $p_{n, 0} = 1$ and $p_{n, 1} = -(2\sqrt{M}+1)$.
\end{enumerate}
\end{lemma}

\begin{proof}[Proof of (1)]

\

Note that for $P$ to have two roots so deep into the disc we must have $p_0 = \pm 1$. Possibly multiplying the series by $-1$ we may assume $p_0 = 1$. 

Let $Q_1(z) = \frac{P(z)}{1-z/\lambda_1}$ and let $Q(z) = \frac{Q_1(z)}{1-z/\lambda_2} = \frac{P(z)}{(1-z/\lambda_1)(1-z/\lambda_2)}$. Since $P(z)$ has roots at $\lambda_1$ and $\lambda_2$ and is holomorphic on $\mathbb{D}$ we know that $Q_1(z)$ and $Q(z)$ are both also holomorphic on $\mathbb{D}$. If we write out their power series expansions $Q_1(z) = \sum_{j=0}^\infty q_{1, j}z^j$ and $Q(z) = \sum_{j=0}^\infty q_jz^j$ then the $q_j$ they satisfy the 2-step recursion:
\begin{align*}
q_0 &= p_0 \\
q_1 &= p_1 + 2(\sqrt{M}+1)\cos(\theta)q_0 \\
q_{j+2} &= p_{j+2} + 2(\sqrt{M}+1)\cos(\theta)q_{j+1} - (\sqrt{M}+1)^2q_j
\end{align*}

where $\cos(\theta)$ is the argument of $\lambda_1$. Note that since $\lambda_1$ could be real it may be that $\cos(\theta) = 1$. Also, note that the escape threshold of $Q(z)$ is exactly 1.

Notice $p_0 = q_{0, 0} = q_0 = 1$ which is exactly equal to the escape threshold for $Q(z)$. This implies that $|q_j| = 1$ for every index $j$. Since $q_j$ is real this means $q_j \in \{-1, 1 \}$. If there is an index $j$ such that $q_j = 1$ and $q_{j+1} = -1$ then $|q_{j+2}|$ will be $> 1$ giving a contradiction, therefore since $q_0 = 1$ we have $q_j = 1$ for every $j$ and hence that $p_j = M$ for every $j \geqslant 2$. Thus $Q(z) = \sum_{j=0}^\infty z^j$ and hence $P(z) = (1-z/\lambda_1)(1-z/\lambda_2)\sum_{j=0}^\infty z^j$. On the other hand, we know that $p_0 = 1$ and $p_2 = p_3 = ... = M$. This is only possible if $\lambda_1 = \lambda_2 = \frac{1}{\sqrt{M}+1}$ and if $p_1 = -(2\sqrt{M} + 1)$ as claimed.

\bigskip
\noindent \textit{Proof of (2)}

Suppose $P_n(z) = \sum_{j=0}^\infty p_{n, j}z^j$ are power series with $p_{n, j} \in S$ and with roots $\lambda_{n, 1}$ and $\lambda_{n, 2}$ converging to $\frac{1}{\sqrt{M} + 1}$ and that $p_{n, 0} = 1$ and $p_{n, 1} = -(2\sqrt{M} + 1)$. Let $P(z) = 1-(2\sqrt{M}+1)z + \sum_{j=2}^\infty Mz^j$. Then we know that:
\begin{align*}
|P(\frac{1}{\sqrt{M} + 1}) - P_n(\frac{1}{\sqrt{M}+1}) | \rightarrow 0.
\end{align*}

Let $D_n(z) = P(z) - P_n(z) = \sum_{j=2}^\infty (M-p_{n, j})z^j$ Then we have $D_n(\frac{1}{\sqrt{M}+1}) \rightarrow 0$. Note that $D_n(z) > 0$ for every positive real number $z$.

If there is some index $j$ such that for infinitely many $n$, $p_{n, j} \neq M$ then $D_n(\frac{1}{\sqrt{M}+1}) \geqslant (M-p_{n, j})(\frac{1}{\sqrt{M}+1})^j$ which means infinitely many $D_n(\frac{1}{\sqrt{M}+1})$ have a uniform non-zero lower bound and hence that this can not converge to $0$.

Thus we must have that for every index $j$ there exists an $N$ such that $n > N$ implies for every $2 \leqslant k < n$, $p_{k, j} = M$.

\bigskip
\noindent \textit{Proof of (3)}

If $|p_{n, 0}| > 1$ then $P_n(z)$ can not have two different roots both deeper then $\frac{1}{\sqrt{M/p_0}+1}$ and hence they can not approach $\frac{1}{\sqrt{M}+1}$. Thus we can assume $|p_{n, 0}| = 1$ for $n$ large enough and hence, after possibly multiplying by $-1$, that $p_{n, 0} = 1$.

If $p_{n, 1} \neq -(2\sqrt{M}+1)$ for all $n$ then because $p_{n, 1} \in S$ and $S$ is finite we must have $|p_{n, 1} + (2\sqrt{M}+1)|$ is uniformly $> 0$.

if $p_{n, 1} < -(2\sqrt{M}+1)$ for infinitely many $n$ then for such $n$ and for any positive real number $z$, $P_n(\frac{1}{\sqrt{M}+1})$ is uniformly $< 1 -(2\sqrt{M}+1)\frac{1}{\sqrt{M}+1} + \sum_{j=2}^\infty M(\frac{1}{\sqrt{M}+1})^j = 0$ and hence $P_n(\frac{1}{\sqrt{M}+1})$ does not converge to $0$ which is a contradiction.

Now suppose $p_{n, 1} > -(2\sqrt{M}+1)$ for infinitely many $n$. Again this means that $2\sqrt{M}+1 + p_{n, 1}$ is uniformly $> 0$ for such $n$. Let $\lambda_n$ be a complex root of $P_n$ in the first quadrant such that $\lambda_n \rightarrow \frac{1}{\sqrt{M}+1}$. Then we have $\lambda_n = \frac{1}{r_n}e^{i\theta_n}$ where $\theta_n \rightarrow 0$ and $r_n \rightarrow \sqrt{M}+1$.

Consider the 2-step recursion:
\begin{align*}
p_{n, 0} &= q_{n, 0}\\
p_{n, 1} &= q_{n, 1} - 2r_n\cos(\theta_n)q_{n, 0} \\
p_{n, j+2} &= q_{n, j+2} - 2r_n\cos(\theta_n)q_{n, j+1} + r^2_nq_{n, j}
\end{align*}

We know that $p_{n, 0} = 1$ and hence $q_{n, 0} = 1$. Then $q_{n, 1} = p_{n, 1} + 2r_n\cos(\theta)$. Notice that as $n \rightarrow \infty$ we have $2r_n\cos(\theta) \rightarrow 2\sqrt{M} + 2$. On the other hand, $p_{n, 1}$ is uniformly $> -(2\sqrt{M} + 1)$ meaning $q_{n, 1}$ is uniformly $> 1$. But the escape threshold for the $q_{n, j}$ tends to $1$ as $n \rightarrow \infty$ giving a contradiction.

Since, for infinitely many $n$ having $p_{n, 1} < -(2\sqrt{M}+1)$ or $p_{n, 1} > -(2\sqrt{M}+1)$ are both impossible, we must have that for $n$ large enough, $p_{n, 1} = -(2\sqrt{M}+1)$ as wanted.

\end{proof}

\begin{theorem}
\label{rho-out}
Let $S$ be a finite non-empty symmetric normalized set of real numbers and let $M = \max(S)$. Then there is no sequence of numbers $\lambda_n \in \Sigma_S \setminus \mathbb{R}$ such that $|\lambda_n| \rightarrow \frac{1}{\sqrt{M}+1}$. Moreover, we have that:

$$\frac{1}{\sqrt{M}+1} < \rho_{out}(S) \leqslant \frac{1}{\sqrt{M+k}} $$
for $k = \max \{ s \in S \mid 1 \leqslant s < 2\sqrt{M} + 1 \}$.
\end{theorem}

\begin{proof}
Arguing by contradiction suppose there was such a sequence $\lambda_n$. By the symmetry of $\Sigma_S$ we may assume $\lambda_n$ are all in the first quadrant. Then since $\Sigma_S$ is compact, possibly passing to a subsequence, we must have that $\lambda_n$ converges to some $\lambda \in \Sigma_S$. By $(1)$ of the previous proposition it follows that $\lambda_n \rightarrow \frac{1}{\sqrt{M}+1}$. 

Let $P_n(z) = \sum_{j=0}^\infty p_{n, j} z^j$ be power series with $p_{n, j} \in S$ such that $P_n(\lambda_n) = 0$. By $(3)$ we have that for $n$ large enough, $p_{n, 0} = 1$ and $p_{n, 1} = -(2\sqrt{M}+1)$. Note that if $-(2\sqrt{M}+1)$ is not in $S$ we already have a contradiction. In particular, we must have that $M \geqslant 2\sqrt{M}+1$.

By $(2)$ we have that for any index $m$ there is an $N$ such that $n > N$ implies $p_{k, j} = M$ for $2 \leqslant k \leqslant m$. 

Consider a power series $P(z) = 1 - (2\sqrt{M}+1)z + \sum_{j=2}^m Mz^j + \sum_{j=m+1}^\infty p_jz^j$ for some $m \in \mathbb{N}$ and some choice of $p_j \in S$. Alternatively, $P(z) = 1 - (2\sqrt{M}+1)z + \sum_{j=2}^\infty Mz^j - \sum_{j=m+1}^\infty (M-p_j)z^j$. If we have $|z| < 1$ then $P(z) = \frac{(1-(\sqrt{M}+1)z)^2}{1-z} - \sum_{j=m+1}^\infty (M-p_j)$. As $p_j \in S$ we have $\sum_{j=m+1}^\infty (M-p_j)z^j$ is positive for all $z > 0$.  Since $\frac{(1-(\sqrt{M}+1)z)^2}{1-z}$ has a root at $z = \frac{1}{\sqrt{M}+1}$ we have that $P(z)$ is negative at $\frac{1}{\sqrt{M}+1}$ whereas it is positive at $z=0$, thus $P(z)$ has a real root in $[0, \frac{1}{\sqrt{M}+1}]$. Moreover, by taking $m$ large we can make the real root as arbitrarily close to $\frac{1}{\sqrt{M}+1}$.

All the previously mentioned $P_n(z)$ are of the form of $P(z)$ above and therefore, by taking $n$ large enough we know $P_n(z)$ has a real root less than $\frac{1}{\sqrt{M}+1}$ as we want. If $\lambda_n$ is complex then $P_n(z)$ also $\overline{\lambda_n}$ approaching $\frac{1}{\sqrt{M}+1}$ meaning the $P_n(z)$ have three roots all approaching $\frac{1}{\sqrt{M}+1}$. But this is impossible as no such power series can have three roots all smaller than $\frac{1}{M^{1/3}+1}$ giving a contradiction. Therefore for $n$ large enough, $\lambda_n$ is real.

Since $\lambda_n$ will be real for $n$ large it follows that there are no complex numbers in $\Sigma_S$ which approach $\frac{1}{\sqrt{M}+1}$ and hence $\rho_{out}(S)$ is strictly larger than $\frac{1}{\sqrt{M}+1}$ giving the first half of the inequality.

For the upper bound on $\rho_{out}(S)$ fix $k$ such that $1 < k < 2\sqrt{M}+1$ and consider the power series $f(z) = 1 - kz + Mz^2+Mz^3 + ...$. For $|z| < 1$ we have that $f(z) = \frac{1-(k+1)z + (M+k)z^2}{1-z}$. Since $1 < k < 2\sqrt{M}+1$ we have that the roots of $f(z)$ are complex and their absolute value is $\frac{1}{\sqrt{M+1}}$.

\end{proof}

The proof above heavily relied on the fact that $S$ was a finite set. If we allow $S$ to be infinite it is easy to show the annulus is best possible. For example, if $S = [-100, -1] \cup [1, 100]$ then we can take a sequence of power series $P_n(z) = 1 - k_nz + \sum_{j=2}^\infty 100z^j$ for $k_n = 20 + \frac{n}{n+1}$. Since the tail is geometric, for $|z| < 1$ this is the same as $1 - k_nz + \frac{100z^2}{1-z} = \frac{1 - (k_n + 1) + (100 + k_n)z^2}{1-z}$. Since $k_n < 21$ the quadratic polynomial $1 - (k_n+1)z + (100+k_n)z^2$ has complex roots and so their absolute value is $\frac{1}{\sqrt{100+k_n}}$. As $n \rightarrow \infty$ we have $k_n \rightarrow 21$ and so the absolute value approaches $\frac{1}{11}$ as wanted.

\subsection{Depth At Each Angle}

We now wish to refine our estimate by fixing a specific angle of interest and studying the depth of roots only at that angle. This subject has been studied carefully in \textbf{[BBBP]} for the case of infinite sets $S$. Specifically, for $S$ a non-empty symmetric normalized set of real numbers we wish to estimate:
\begin{align*}
\rho_{\theta}(S) = \inf \left\{ R > 0 \mid Re^{i\theta} \in \Sigma_S \right\}.
\end{align*}

As $S$ is symmetric we can restrict our focus to $\theta \in [0, \pi/2]$. We already know that $\rho_{\pi/2} = \frac{1}{\sqrt{M+1}}$ and that $\rho_0 = \frac{1}{M+1}$ we really we want to know for angles in $(0, \pi/2)$.

\begin{definition}
Let $S$ be a non-empty symmetric normalized set of real numbers bounded above. Let $M = \sup_{s \in S} |s|$ and $\lambda = \frac{1}{r}e^{i\theta}$ for $r > 1$ and $\theta \in (0, \pi/2]$ such that $r^2-2r\cos(\theta) \neq 1$. We define the angled escape threshold $E = E(S, \lambda)$ as $E = \frac{M}{r^2-2r\cos(\theta)-1}$.
\end{definition}

\begin{lemma}
Let $S$ be a non-empty symmetric normalized set of real numbers bounded above. Let $\lambda = \frac{1}{r}e^{i\theta} \in \mathbb{D}$ with $\theta \in (0, \pi/2]$ such that $r^2-2r\cos(\theta) > 1$. If $p_j \in S$ for each $j$ and $q_j$ are defined recursively from $p_j$ and $\lambda$ as in the 2-step then $\lambda$ is a root of $\sum_{j=0}^\infty p_jz^j$ if and only if $|q_j| \leqslant E(S, \lambda)$ for each $j$.
\end{lemma}

\begin{proof}
The reverse implication follows from the regular 2-step since $|q_j| \leqslant E(S, \lambda)$ implies that $|q_j|$ is bounded.

Suppose that for some choice of $(p_j)_{j=0}^\infty$ in $S$ and for some index $j$ we had $|q_j| > E(S, \lambda)$. Then we have:

\begin{align*}
p_{j+2} = q_{j+2} - 2r\cos(\theta)q_{j+1} + r^2q_j.
\end{align*}

Suppose that $|q_{j+1}|$ and $|q_{j+2}|$ are both $< E(S, \lambda)$. Then:

\begin{align*}
|p_{j+2}| > (r^2-2r\cos(\theta)-1)E(S, \lambda) = M
\end{align*}

Which gives a contradiction. Thus at least one of $q_{j+1}$ and $q_{j+2}$ must have absolute value at least $E(S, \lambda)$. This implies that $\sum_{j=0}^\infty p_jz^j$ is not a polynomial and hence $\lambda \notin \Sigma_{S}^{(fin)}$. On the other hand, we know that any element $\lambda \in \Sigma_S$ can be approximated by elements $\lambda_k = \frac{1}{r_k}e^{i\theta_k} \in \Sigma_{S}^{(fin)}$. In particular, we can approximate $\lambda$ by some choice of roots $\lambda_k$ of the parital sums $P_k(z) = \sum_{j=0}^k p_jz^j$. As these are polynomials with coefficients in $S$, we have that the $q_j(\lambda_k)$ attained from applying the 2-step recursion to $P_k$ with $\lambda_k$ will satisfy $|q_j(\lambda_k)| \leqslant E(S, \lambda_k)$. For each fixed index $j$ we know that $q_j(\lambda)$ is a continuous function of $\lambda$ and that $E(S, \lambda)$ is a continuous function of $\lambda$. Thus, since $|q_{j, k}| \leqslant E(S, \lambda_k)$ for each $k$ and each $j$, we have that $|q_j| \leqslant E(S, \lambda)$ for each $j$ as claimed.

\end{proof}

\begin{corollary}
Let $S$ be a non-empty symmetric normalized set of real numbers bounded above. Let $\lambda = \frac{1}{r}e^{i\theta} \in \mathbb{D}$ with $r > 1$ and $\theta \in (0, \pi/2]$. If $0 < E(S, \lambda) < 1$ then $\lambda \notin \Sigma_S$. In particular, we have that:
\begin{align*}
\rho_{\theta}(S) \geqslant \frac{1}{\cos(\theta) + \sqrt{\cos(\theta) + M + 1}}.
\end{align*}
\end{corollary}

\begin{proof}
For a fixed $\theta$ we have that when $r > \cos(\theta) + \sqrt{\cos(\theta)+M+1}$, $r^2-2r\cos(\theta)-1$ will be $> M$ and hence $0 < \frac{M}{r^2 - 2r\cos(\theta) -1} < 1$. Thus by the previous lemma, $\frac{1}{r}e^{i\theta} \notin \Sigma_S$.
\end{proof}

When $\theta = \pi/2$ the above estimate gives $\rho_\theta = \frac{1}{\sqrt{M+1}}$ which is best possible. On the other hand, as $\theta$ approaches $0$ the estimate tends to $\frac{1}{\sqrt{M+2}+1}$ which is much larger then the estimate $\frac{1}{\sqrt{M}+1}$ we saw before. Thus for angles around $\pi/2$ this estimate is better but for angles close to $0$ it is worse. As $M$ grows larger this estimate will be better for larger intervals of angles.

\section{Connectedness}

The goal of this section is to prove:

\begin{theorem}
\label{conn-crit}
If $S$ is a normalized symmetric set of integers containing $0$ then:
	\begin{enumerate}
		\item If the total gap of $S$ is $3$ or more then $\Sigma_S$ is disconnected.
		\item If the total gap of $S$ is $1$ or $2$ and if $M = \max(S) \geqslant 40$ then $\Sigma_S$ is connected and locally connected.
	\end{enumerate}
\end{theorem}

\begin{figure}[h]
\centering
\includegraphics{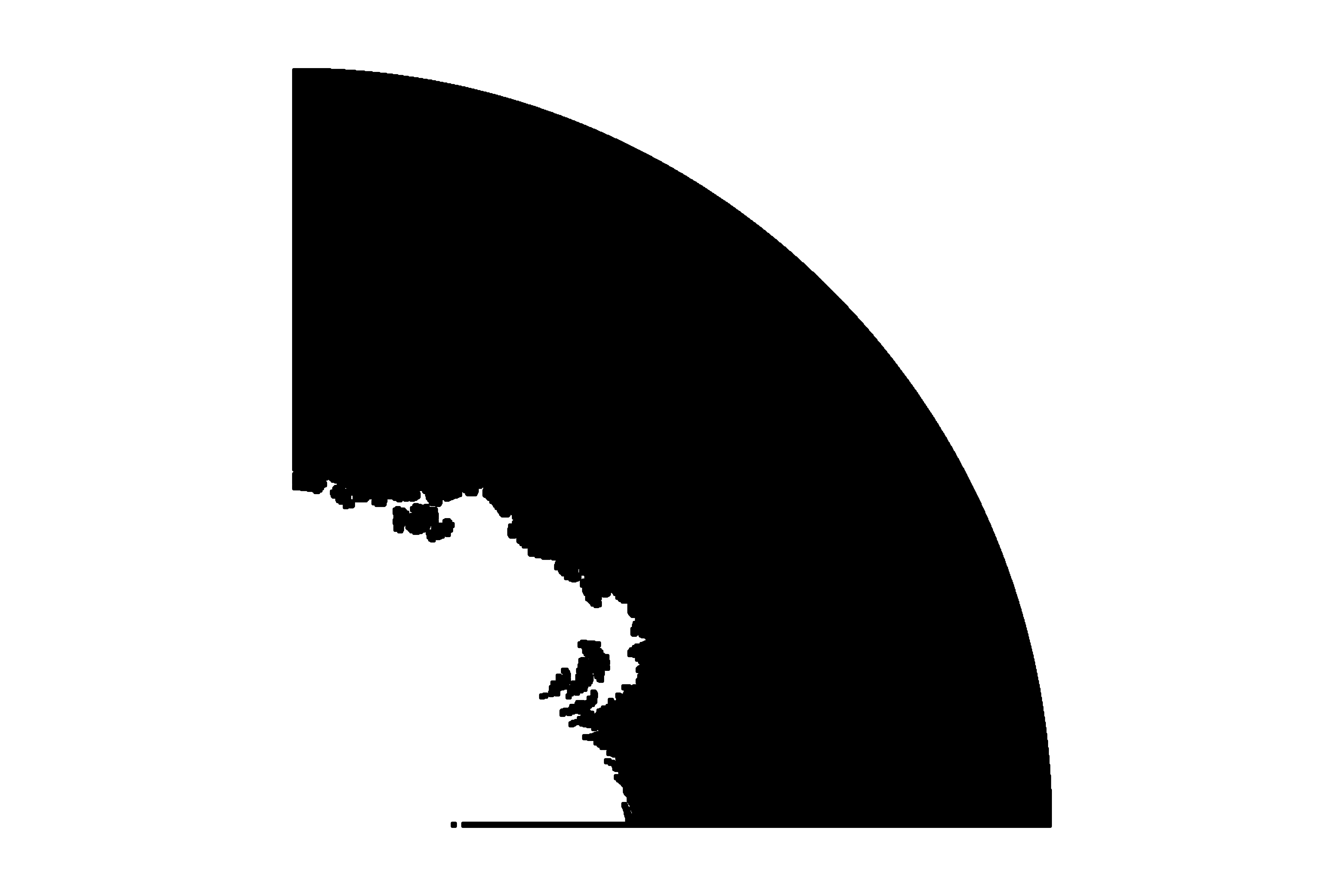}
\caption{$\Sigma_{\{0, \pm 1, \pm 4 \} }$ in the first quadrant.}
\end{figure}

We begin by considering the following subsets of $\Sigma_S$:
\begin{align*}
\Sigma_S^{1} &= \left\{ z \in \mathbb{D} \mid \exists p_1, p_2, ... \in S, 1+ \sum_{j=1}^\infty p_jz^j = 0 \right\}.
\end{align*}

\begin{proposition}
If $S$ is a symmetric non-empty finite subset of real numbers and the total gap of $S$ is at least $3$ then $\Sigma_S \cap \mathbb{R_+}$ and $\Sigma_S^1 \cap \mathbb{R_+}$ are both disconnected. In particular, both $\frac{1}{M}$ and $\frac{1}{M+1}$ are in $\Sigma_S$ and $\Sigma_S^1$ however there is a real number between them not in $\Sigma_S$ nor in $\Sigma_S^1$.
\end{proposition}

\begin{proof}
Without loss of generality assume $S$ is normalized. Fix consecutive positive elements $s, s'$ in $S$ with $s'-s \geqslant 3$ and let $r\in (0, 1)$ such that $(M+r)r = s + 3/2$. Then for every element $t \in S$ we have $|(M+r)r-t| \geq 3/2$. 

To prove the theorem we first show that $\frac{1}{M+r} \notin \Sigma_S$. To do so notice that the escape threshold for the recursion on the root $\frac{1}{M+r}$ is $\frac{M}{|M+r|-1} < 3/2$ since $M \geqslant 4$ (as otherwise it is not possible to have a total gap of $3$). As a result, if the $q_j$ associated to a power series $P$ ever reach $3/2$ then we know $\frac{1}{M+r}$ is not a root of $P$.

Consider a power series $P(z) = \sum_{j=0}^\infty p_jz^j$ with $p_0 \neq 0$ and $p_j \in S$. We know that $|p_0| \geqslant 1$ and hence that $|q_0 \geqslant 1$ also. Thus $|q_1| = |p_1 + (M+r)q_0| \geqslant M+r - |p_1| \geqslant r$. Then the next step in the recursion gives $|q_2| = |p_2 - (M+r)q_1| \geqslant |(M+r)r - |p_2|| = |s+3/2 - |p_2|| \geqslant 3/2$. Thus we see that the $q_j \rightarrow \infty$ and hence $\frac{1}{M+r}$ is not a root of $P(z)$. This works for every $P(z)$ with $p_j \in S$ and so $\frac{1}{M+r} \notin \Sigma_S$. 

Finally note that both $\frac{1}{M+1}$ and $\frac{1}{M}$ both belong to $\Sigma_S$ (the first being a root of $-1 + Mz + Mz^2 +...$ and the second a root of $1-Mz$) while $\frac{1}{M+r}$ is between them and so the positive real part of $\Sigma_S$ is disconnected.

Since the power series making $\frac{1}{M+1}$ and $\frac{1}{M}$ belong to $\Sigma_S$ begin with $1$ the result also holds for $\Sigma_S^1$.

\end{proof}

\begin{corollary}
If $S$ is a normalized non-empty finite subset of real numbers and the total gap of $S$ is at least $3$ then $\Sigma_S$ is disconnected.
\end{corollary}

\begin{proof}
This follows from the above proposition that $\mathbb{R} \cap \Sigma_S$ is disconnected containing $\frac{1}{M}$ and $\frac{1}{M+1}$ but missing a real number between them. Furthermore by definition we know that $0 \notin \Sigma_S$. Thus the only way for $\Sigma_S$ to be connected is if there were a path of complex numbers in $\Sigma_S$ connecting $\frac{1}{M+1}$ to the rest of the set. 

Recall the previous estimate that complex roots live in the annulus $\frac{1}{\sqrt{M}+1} \leqslant |z| < 1$. Since the total gap of $S$ is at least $3$ and $1 \in S$ this means that $M$ is at least $4$ and so $\frac{1}{M} < \frac{1}{\sqrt{M}+1}$ which shows that no such path can exist and hence $\Sigma_S$ and $\Sigma_S^1$ are both disconnected.
\end{proof}

Turning now to the question of connectedness we rely on a criterion for $\Sigma_S^1$ to be connected (and locally connected) found in \textbf{[Na]} as follows. Define $\Delta S = \{ a - b : a, b \in S \}$ and construct a graph with vertices $S$ and an edge $(a, b)$ if $(a-b)S \subseteq \Delta S$. Then \textbf{[Na] theorem B} says that if this graph is connected and if $\Sigma_S^1$ contains an annulus $L < |z| < 1$ for any $L \in (0, 1)$ then $\Sigma_S^1$ is connected and locally connected.

\begin{proposition}
If $S$ is a normalized finite set of integers containing $0$ and the total gap of $S$ is $1$ or $2$ then $\Sigma_S^1$ is connected and locally connected
\end{proposition}

\begin{proof}

The fact that $0 \in S$ implies that $1 \cdot S = S \subseteq \Delta S$. As $S$ is symmetric we have that $2 \cdot S \subseteq \Delta S$ also. It follows that the graph defined above is connected. Moreover we know that $\Sigma^1_S$ always contains an annulus $\frac{1}{\sqrt{2}} < |z| < 1$ (because $\Sigma_{0, \pm 1}$ does) and so we have that $\Sigma^1_S$ is connected and locally connected. 

\end{proof}

Combining these propositions gives a full classification of the connectedness of $\Sigma_S^1$.

\begin{theorem}
If $S$ is a normalized symmetric set of integers containing $0$ then:

	\begin{enumerate}
		\item If the total gap of $S$ is $1$ or $2$ then $\Sigma_S^1$ is connected and locally connected.
		\item Otherwise the total gap of $S$ is at least $3$ and $\Sigma_S^1$ is disconnected.
	\end{enumerate}
\end{theorem}

Notice that the disconnected part of the proof also held for $\Sigma_S$ whereas the connectedness proof required the power series to start at $1$. Below we try to extend this criterion to $\Sigma_S$ by proving that if $M = \max(S)$ is large enough then $\Sigma_S = \Sigma_S^1$.

\begin{proposition}
If $S$ is a symmetric normalized set of integers with total gap at most $2$ and $M = \max{S} \geqslant 40$ then $\Sigma_S^1$ contains the annulus $\frac{1}{\sqrt{\frac{M}{2}} + 1} \leqslant |z| < 1$.
\end{proposition}

\begin{proof}
Let $\lambda \in \frac{1}{\sqrt{\frac{M}{2}} + 1} \leqslant |z| < 1$. If $\lambda \in \mathbb{R}$ then $\lambda \in \Sigma_S^1$.

If $\lambda = \frac{1}{r}e^{i\theta}$ is not real then we can use the 2-step. Choose $p_0 = q_0 = 1$. Then we have that:
\begin{align*}
q_1 = p_1 + 2r\cos(\theta).
\end{align*}

We have that $|2r\cos(\theta)| \leqslant 2(\sqrt{M/2} + 1)$. Since $M \geqslant 40$ this is $\leqslant M + 1$ and hence we can pick $p_1$ such that $|q_1| \leqslant 1$.

Suppose now we have $q_n$ and $q_{n+1}$ both in $[-1, 1]$. Then we have that:
\begin{align*}
q_{n+2} = p_{n+2} + 2r\cos(\theta)q_{n+1} - r^2q_n.
\end{align*}

We have that $|2r\cos(\theta)q_{n+1} - r^2q_n| \leqslant 2(\sqrt{M/2} + 1) + (\sqrt{M/2} + 1)^2 = M/2 + 4\sqrt{M/2} + 3$. Since $M \geqslant 40$ this is $\leqslant M+1$ and so we can choose $p_{n+2} \in S$ such that $|q_{n+2}| \leqslant 1$.

\end{proof}

\begin{corollary}
If $S$ is a symmetric normalized set of integers with total gap at most $2$ and $M = \max{S} \geqslant 40$ then $\Sigma_S = \Sigma_S^1$.
\end{corollary}

\begin{proof}
Clearly $\Sigma_S^1 \subseteq \Sigma_S$ so it suffices to prove that $\Sigma_S \subseteq \Sigma_S^1$. Suppose $\lambda \in \Sigma_S$ is the root of a power series $\sum_{j=0}^\infty p_jz^j$ with $p_j \in S$. If $p_0 = \pm 1$ then $\lambda \in \Sigma_S^1$ and we are done. Also since the total gap of $S$ is $\leqslant 2$ we know that $\mathbb{R} \cap \Sigma_S = \mathbb{R} \cap \Sigma_S^1 = (-1, \frac{-1}{M+1}] \cup [\frac{1}{M+1}, 1)$ and so if $\lambda \in \mathbb{R}$ then $\lambda \in \Sigma_S^1$ as well. Otherwise we know that $|p_0| \geqslant 2$ and that it is imaginary. We know that any complex root of such a power series lives in the annulus $\frac{1}{\sqrt{M/p_0} + 1} \leqslant |z| < 1$ and hence in $\frac{1}{\sqrt{M/2} + 1} \leqslant |z| < 1$. Then by the previous proposition we have $\lambda \in \Sigma_S^1$ as wanted.
\end{proof}

Combining this corollary with the classification of the connectedness of $\Sigma_S^1$ we obtain \textbf{Theorem \ref{conn-crit}} as wanted.

Notice that there are only finitely many normalized symmetric sets of integers $S$ containing $0$ with $\max(S) < 40$ and so this gives theorem works for all but finitely many such $S$.

To improve the argument above it would help to better understand how deep complex roots of power series with real coefficients can get into $\mathbb{D}$. The estimate that the complex roots of $\sum_{j=0}^\infty p_jz^j$ must live in $\frac{1}{\sqrt{M/p_0} + 1} \leqslant |z| < 1$ is not optimal when $p_j$ all live in a finite set. If this bound could be improved the above argument would generalize for smaller $M$ values.

It would also be ideal to have a better understanding of how (and if) $\Sigma_S$ and $\Sigma_S^1$ differ.

Lastly, it would be of great interest to have a better understanding of the role that $0$ being in $S$ plays in the connectedness of $\Sigma_S$.

\section{Rigidity}

Finally we turn to the question of what information can be derived from knowing that $\Sigma_S = \Sigma_T$. For this discussion we mostly restrict to considering normalized symmetric sets of integers. This is more restrictive than it appears at first glance. For example, if $S = \{ 2, 3, 4\}$ then $S$ is a set of integers but is not normalized. We can normalize $S$ by dividing by $2$ to get $S' = \{ 1, 3/2, 2 \}$ however $S'$ is no longer a set of integers. 

\begin{proposition}
If $S$ and $T$ are finite non-empty normalized sets of real numbers and if $\Sigma_S = \Sigma_T$ then $\max_{s \in S} s = \max_{t \in T} t$.
\end{proposition}

\begin{proof}
Since $S$ and $T$ are symmetric we know that the annulus described in \textbf{Lemma \ref{outerannulus}} is best possible meaning that the size of the annulus containing $S$ is equal to the size of the annulus containing $T$. The annuli is proportional to the min (positive non-zero) and max elements of $S$ and $T$ respectively and so the ratios must agree for $\Sigma_S$ to be equal to $\Sigma_T$. Since both sets are normalized the minimum is $1$ and so the maximum elements must agree.
\end{proof}

Because the maximum elements are equal we can denote the shared max by $M$.

Next we prove one of the major results of the paper, that if two root sets are equal then their early terms must agree, thus the smaller elements of the set of coefficients are rigid. 

\begin{theorem}
\label{weakrigidity}
Let $S$ and $T$ be finite non-empty symmetric normalized sets of integers. If $\Sigma_S = \Sigma_T$ and $k$ is an integer in $(1, 2\sqrt{\max(S)}+1)$ then $k\in S$ if and only if $k \in T$.
\end{theorem}

\begin{proof}
Let $M = \max(S)$. Since $\Sigma_S = \Sigma_T$ we also have that $M = \max(T)$. Consider the polynomial $f(z) = 1 - (k+1)z + (M+k)z^2$. Note that the assumption $1 < k < 2 \sqrt{M}+1$ gives that the roots of $f$ are not real and hence allows us to use the 2-step recursion to say that its roots belong to $\Sigma_S$ if and only if we have a sequence $(p_j)$ in $S$ so that there is a uniformly bounded solution $(q_j)$ to the recursion:
\begin{equation}
\label{rec2}
p_0 = q_0, \quad p_1 = q_1 - (k+1)q_0, \quad p_{j+2} = q_{j+2} - (k+1)q_j + (M+k)q_j.
\end{equation}

Our claim is that such a solutions exists if and only if $k \in S$ which would immediately imply the theorem.

To prove the sufficiency we consider the power series:
\begin{align*}
1 - kz + Mz^2 + Mz^3 + ... &= 1 - kz + M\frac{z^2}{1-z}.
\end{align*}

Setting it equal to zero we see that:
\begin{align*}
& 0 = 1 - kz + M\frac{z^2}{1-z} \\
\iff & 0 = 1(1-z) - k(1-z) + Mz^2 \\
\iff & 0 = 1 - (k+1)z + (M+k)z^2.
\end{align*}

The above calculation only works for $|z| < 1$ but this is not an issue as the roots of $f$ will satisfy $|z| = \frac{1}{\sqrt{M+k}}$ and hence are both $< 1$. Since the power series only used $1, -k,$ and $M$ we see that the roots of $f$ are in $\Sigma_S$ when $k \in S$.

\bigskip
\bigskip

For the necessity, suppose that we have some sequence $(p_j)$ with integer values in $\{ 0, \pm 1, \pm 2, ..., \pm M \}$ such that the corresponding sequence $(q_j)$ is uniformly bounded and satisfies \eqref{rec2}. First note that each $q_j$ must be an integer. Define a sequence $(h_j)$ by $h_0 = q_0$ and $h_{j+1} = q_{j+1} - q_j$. Since $(q_j)$ is a uniformly bounded sequence of integers so is $(h_j)$. Moreover we have that $q_j = \sum_{i=0}^j h_i$. Let $A = \max_{j} |q_j|$ be achieved at index $J$. Possibly multiplying by $-1$ we can assume $q_J = A$. Consider the recursion at the $J+2$ step:
\begin{align*}
p_{J+2} = q_{J+2} - (k+1)q_{J+1} + (M+k)q_J.
\end{align*}

Recall that $|p_{J+2}| \leq M$. On the other hand the right hand side has $q_J = A$ (which presumably is a large integer) times $M+k$. We claim that this is possible only if $A$ is 1. To see this expand the $q_j$ terms as sums of $h_i$ to get:
\begin{align*}
p_{J+2} &= \sum_{j=0}^{J+2} h_j - (k+1) \sum_{j=0}^{J+1} h_j + (M+k) \sum_{j=0}^J h_j \\
&= h_{J+2} - kh_{J+1} + Mq_J.
\end{align*}

Now notice that the assumption that $q_J$ is maximal implies that $h_{J+1} \leqslant 0$ as otherwise we would have $q_{J+1} = q_J + h_{J+1}$ is larger than $q_J$. Assuming the sequence $h_j$ is not eventually always zero, if $h_{J+1} = 0$ we can reindex until it is $< 0$ and hence $\leqslant -1$. On the other hand, we have that $|q_{J+2}| = |h_{J+2} + h_{J+1} + q_J| \leqslant A$ and hence that $h_{J+2} \geqslant 1-2A$. Thus
\begin{align*}
p_{J+2} & \geqslant 1-2A + k + MA\\
&\geqslant k + 1 + (M-2)A\\
&\geqslant 1 + k + M - 2 \\
& > M.
\end{align*}

But we know that $|p_{J+2}| \leq M$ (because this is true for every $p_j$) and so we have a contradiction. Thus once $q_J = A$ achieves its maximum value we have $h_{J+1} = h_{J+2} = ... = 0$. But then \eqref{rec2} simplifies to:
\begin{align*}
p_{J+2} &= MA
\end{align*}

which is possible if and only if $A \in \{0, \pm 1 \}$. We already assumed $A \geqslant 0$ so it can not be $-1$. On the other hand it also can not be $0$ because $q_0 = p_0$ and $p_0$ by definition can not be $0$. Thus $A = 1$ and hence $(q_j)_{j=1}^\infty$ is a constant sequence of $1's$. But by tracing the recursion, the only way to have $q_0 = 1$ is for $p_0 = 1$. Then for $q_1 = 1$ we see $p_2 = -k$, and afterwards we get $p_j = M$. Thus $1, -k, M$ were all elements of $S$ and so $k \in S$ as wanted.

\end{proof}

Lastly we show that coefficient sets have a sort of quasi-rigidity, namely that if an element $k$ is in a set of coefficients then for another set of coefficients to give the same roots, it must contain either $k$, $k-1$, or $k+1$. Of course for small elements $k$ this is weaker than the theorem above. The strength of the theorem is that is holds for all elements.

\begin{theorem}
\label{quasirigidity}
If $S$ and $T$ are finite non-empty normalized sets of integers and if $\Sigma_S = \Sigma_T$ then $k \in S$ implies that either $k$, $k-1$, or $k+1$ is in $T$.
\end{theorem}

\begin{proof}
Let $M = \max(S) = \max(T)$. Fix $k \in S$ with $1 < k < M$. To prove this theorem we introduce a number $\sigma_k$ and claim that $\frac{1}{M+\sigma_k} \in \Sigma_S \iff $ at least one of $k$, $k-1$, or $k+1$ belongs to $S$. Let $\sigma_k$ be a number in $(0, 1)$ such that $(M+\sigma_k)\sigma_k = k$. Since the left hand side ranges from $0$ to $M+1$ as $\sigma_k$ ranges from $0$ to $1$ there must be a choice $\sigma_k$ satisfying the equality. Consider the recursion:
\begin{align*}
p_0 = q_0 \quad p_{j+1} = q_{j+1} - (M+\sigma_k) q_j.
\end{align*}

We first show that there is a choice of $p_j$ such that the $q_j$ are uniformly bounded provided $k, k-1,$ or $k+1$ is in $S$. Taking $p_0 = 1$ we see $q_0 = 1$ and so $p_1 = q_1 - M - \sigma_k$. If $p_1 = -M$ then $q_1 = \sigma_k$ and so the next recursion is $p_2 = q_2 - k$. If we assume $k \in S$ then taking $p_2 = -k$ gives $q_2 = 0$ from which point we can just take $p_j = 0$ to get $q_j = 0$ and are done (we have constructed a polynomial). If instead $k+1$ or $k-1$ is in $S$ then we can pick $p_2$ so that $q_2 = \pm 1$. If $q_2 = 1$ then we have gone in a loop and can just pick the same choices over and over. If $q_2 = -1$ then we pick the negative choices of those made to get a future $q_j = 1$ at which point we repeat. Thus in all cases we have a uniformly bounded solution $(q_j)$ and so are done.

On the other hand suppose now that none of $k, k-1, k+1$ are in $S$. It should be noted that $M \geqslant 5$ as otherwise it is not possible to not have 3 consecutive terms missing from $S$. Similarly we can also assume $k \geqslant 3$ as we know $1 \in S$. The escape threshold for this recursion is $\frac{M}{M+\sigma_k-1} < \frac{5}{4}$. If we pick $|p_0| \neq 1$ then it must be $\geqslant 2$. Thus we get $|q_2| \geq M + 2\sigma_k > 5$ and which point the $|q_j| \rightarrow \infty$.

If we pick $p_0 = 1$ then we get $p_1 = q_1 - M - \sigma_k$. If $p_1 \neq -M$ then $|q_1| \geqslant 1+\sigma_k$ which implies $|q_2| \geqslant M+k+\sigma_k^2 - M = k + \sigma_k^2 > 3$ at which point $|q_j| \rightarrow \infty$.

Finally suppose now that $p_0 = 1, p_1 = -M$. In this case the recursion says $p_2 = q_2 - k$. If $|p_2| \notin \{k, k-1, k+1 \}$ then $|q_2| \geqslant 2$ and so $|q_j| \rightarrow \infty$. Thus the only way to have a uniformly bounded sequence $(q_j)$ is to pick one of $k, k-1,$ or $k+1$ which completes the proof.

\end{proof}

\begin{corollary}
Let $S, T$ be finite non-empty normalized symmetric sets of integers such that $\Sigma_S = \Sigma_T$. If $a, b \in S$ and $|a-b| > 3$, and there is no $k \in S$ such that $a < k < b$ then $a, b \in T$ and there is no $k \in T$ with $a < k < b$.
\end{corollary}

\begin{proof}
Let $M = \max(S) = \max(T)$. Since $a \in S$ we know that $\frac{1}{M+ \sigma_{a+1}} \in \Sigma_S$ and hence either $a, a+1,$ or $a+2 \in T$. If we had that $a+1 \in T$ or $a+2 \in T$ then we would have $\frac{1}{M+\sigma_{a+2}} \in \Sigma_T$ and hence in $\Sigma_S$ which implies $a+1, a+2,$ or $a+3$ is in $S$. But by assumption we know that $a+1, a+2 a+3$ are all not in $S$. Thus $a \in T$. By similar reasoning we deduce $b \in T$.

If there were some $k$ between $a$ and $b$ which was in $T$ then both $\frac{1}{M+\sigma_{k+1}}$ and $\frac{1}{M+\sigma_{k-1}}$ would be in $\Sigma_S$ at least one of which is impossible.
\end{proof}

\subsection{Geometric Interpretation for the early rigidity}

The theorem on early rigidity was presented previously as a purely analytic result. Here we present a more geometric and heuristic explanation for why one would expect the theorem to be true.

\begin{figure}[!h]
\centering
\includegraphics[scale=0.7]{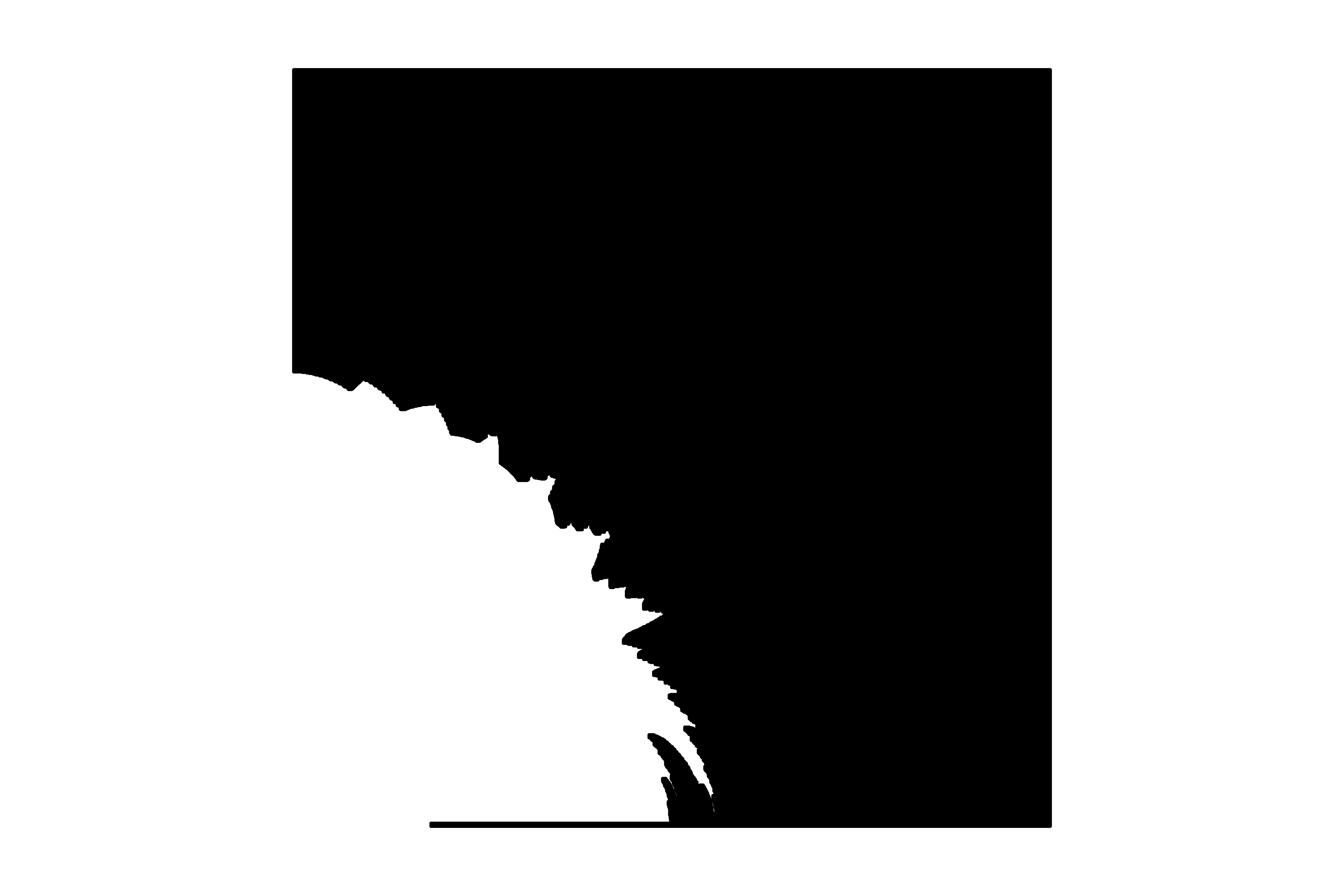}
\caption{The set of roots of power series with coefficients any integer from $-10$ to $10$.}
\label{fig:0 to 10}
\end{figure}

Consider \textbf{Figure \ref{fig:0 to 10}} showing the set of roots of power series with coefficients any integer from $-10$ to $10$. Note that $2\sqrt{10}+1$ is approximately equal to $7.324$. Thus we know that the coefficients $2$ through $7$ are rigid.

\begin{figure}[!h]
\centering
\includegraphics[scale=0.7]{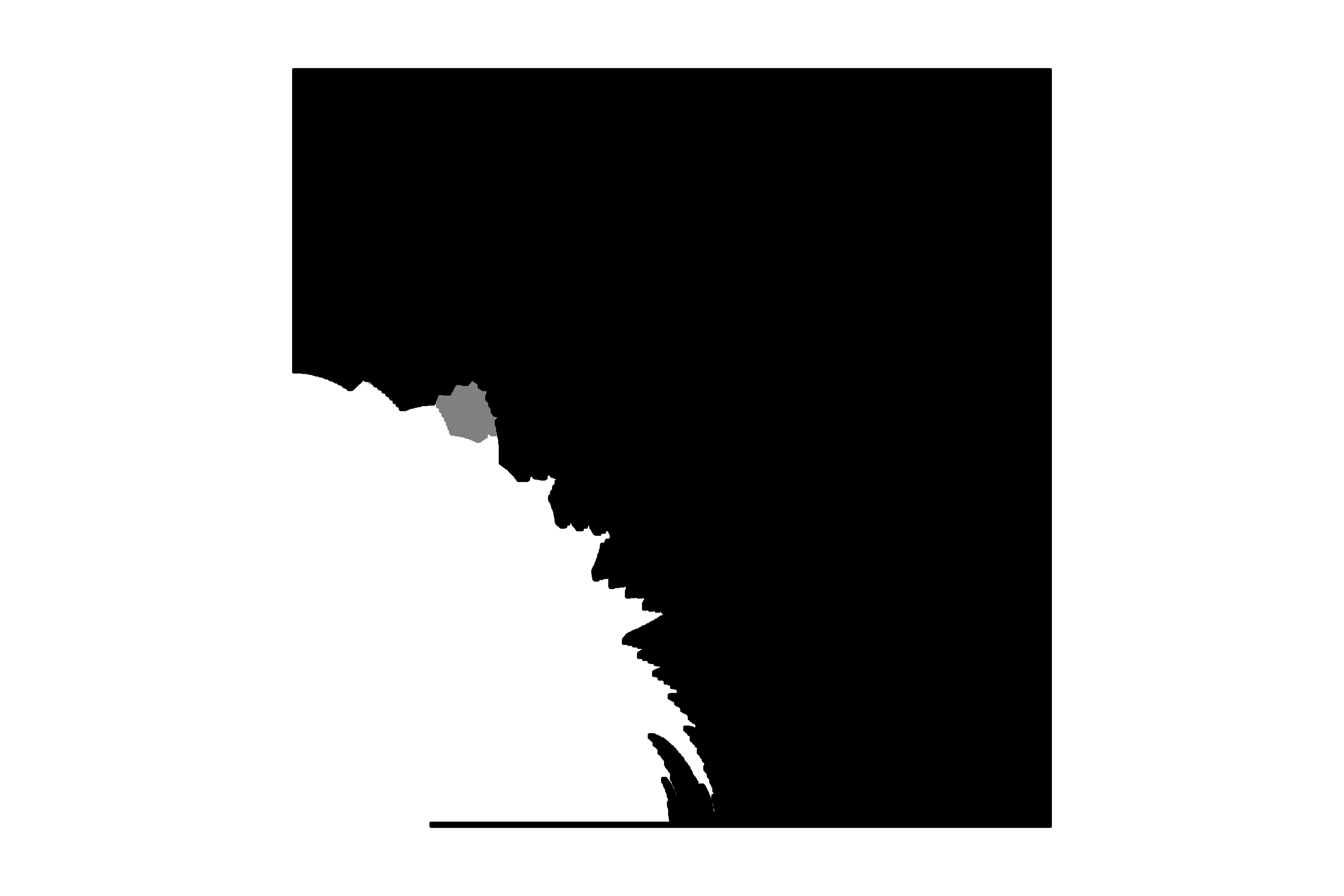}
\caption{In black we see the set of roots of power series with coefficients any integer from $-10$ to $10$ except $\pm 2$. In grey are the roots we get if we also allow $\pm 2$.}
\label{fig:0 to 10 no 2}
\end{figure}

\begin{figure}[!h]
\centering
\includegraphics[scale=0.7]{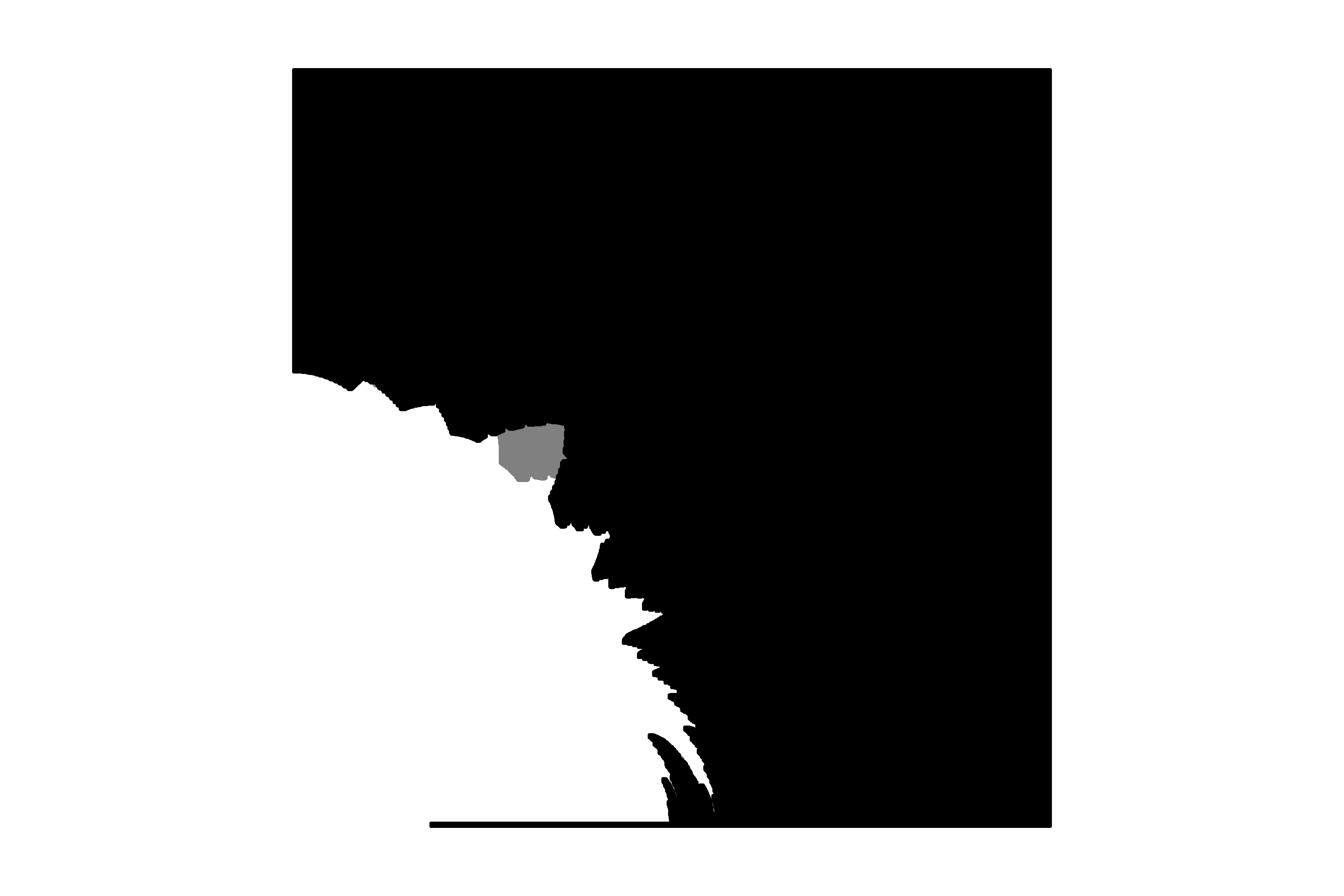}
\caption{In black we see the set of roots of power series with coefficients any integer from $-10$ to $10$ except $\pm 3$. In grey are the roots we get if we also allow $\pm 3$.}
\label{fig:0 to 10 no 3}
\end{figure}

Studying the shape of the set of roots one immediately observes what for a lack of better term we will call `spikes'. Starting from the imaginary axis around $\frac{i}{\sqrt{11}}$ the roots start to get deeper into the unit disc, then recede back, then get deeper, then recede over and over. A quick count shows that there are $8$ such `spikes', two more than the number of digits proven to be rigid. In fact, a less formal understanding of the early rigidity theorem is that the set of roots of power series with coefficients $-M$ through $M$ will always have $\lceil 2\sqrt{M}+1 \rceil$ `spikes'.

To see how, consider \textbf{Figure \ref{fig:0 to 10 no 2}} which shows what roots are lost if we allow all digits from $-10$ to $10$ except $\pm 2$. Two things are immediately apparent. Firstly, one of the `spikes' (the third ordered clockwise starting at $\frac{i}{\sqrt{11}}$) is removed. Secondly, no other region of the set of roots seems to be  significantly affected (if you zoom in enough there are minor differences in other areas). Clearly something about this `spike' in the root set is related to having a $2$ in the coefficient set. This, in a sense, is what we have proven above.

In \textbf{Figure \ref{fig:0 to 10 no 3}} we see the same behaviour if we remove $\pm 3$. The same phenomenon will happen for $\pm 4, \pm 5, \pm 6$, and $\pm 7$.

To understand this phenomenon consider a power series $P(z) = \sum_{j=0}^{\infty} p_j z^j$. How do the various numbers $p_j$ affect the roots of $P(z)$? Since $P(z)$ could have many roots this is not an easy question to answer so we instead ask: how do the various numbers $p_j$ effect extremely small complex roots of $P(z)$? 

Clearly the choice of $p_0$ affects the smallest possible absolute value of the root. In fact, the ratio between $p_0$ and the later choices $p_2, p_3, ...$ more or less determine the smallest possible absolute value. $p_1$ on the other hand plays a rather distinct role. It chooses the angle.

To see why, consider some arbitrary non-real root $\lambda = \frac{1}{r}e^{i\theta}$ of a power series $1 + \sum_{j=1}^\infty p_jz^j$. By the two step recursion from \textbf{Theorem \ref{2-step}} we know that $q_0 = p_0 = 1$ and hence that $q_1 = p_1 + 2r\cos(\theta)q_0 = p_1 + 2r\cos(\theta)$. Since for every $j$, $|q_j| \leqslant E_2 = \frac{M}{(r-1)^2}$ we have that:

\[
\frac{-p_1 - E_2}{2r} \leqslant \cos(\theta) \leqslant \frac{-p_1+E_2}{2r}
\]

giving a bound on $\theta$ dependent on $M$ and on $r$. The above holds for any choice of root however how tight it is depends on $M$ and on $r$. When $r$ is large relative to $M$, say $\sqrt{M} \leqslant r \leqslant \sqrt{M}+1$ then as $M$ grows, $E_2$ approaches $1$ and hence the difference between the bounds behaves like $\frac{1}{r}$ and can be made as close to $0$ as required. This means that for roots that are relatively deep into the unit disc, the $p_1$ coefficient plays a distinguished role in choosing $\cos(\theta)$. Moreover, when $M$ and $p_1$ are fixed, growing $r$ brings the bounds closer together giving fewer possibilities for $\cos(\theta)$ causing this `spike' structure to form.

On the other hand, if we read the same recursion but instead fix $r$, $M$, and $\theta$ then we find that:

\[
-2r\cos(\theta) - E_2 \leqslant p_1 \leqslant -2r\cos(\theta) + E_2.
\]

The difference in these bounds is $2E_2$ so for $r$ and $M$ large enough, there are at most $3$ choices for $p_1$. Further analysis can often disqualify two of these choices (as in \textbf{Theorem \ref{weakrigidity}}) and in practice for $r$ and $M$ large the only viable choice will be the $p_1$ closest to the upper bound.

This also explains why there are more `spikes' then there are rigid coefficients in the argument. Recall that removing $\pm 2$ from the coefficient set was the same as removing the third `spike' from the root set. The first and second `spikes', it turns out, will correspond to having a $0$ or a $\pm 1$ respectively placed in the $p_1$ position. In fact, a similar argument as the early rigidity could be used to show that $\Sigma_{ \{ 0, \pm 1, ..., \pm M \} } \neq \Sigma_{ \{ \pm1, ..., \pm M \}}$ (on the other hand, showing $\Sigma_{S \cup \{ 0 \}} \neq \Sigma_{S \setminus \{ 0 \}}$ for a general set $S$ seems much harder). Clearly $1$ is rigid for normalized sets but an argument like in the early rigidity proof would not make sense for discussing removing $1$.

The relationship between the `spikes' and the choice of $p_1$ is also why there are two more `spikes' than rigid coefficients in the argument. Indeed, if each `spike' really corresponds to a difference choice of $p_1$ then the first and second `spikes' will correspond to taking roots of power series with $p_1 = 0$ and $p_1 = -1$ respectively. In fact, a similar argument as the early rigidity could be used to show that $\Sigma_{ \{ 0, \pm 1, ..., \pm M \} } \neq \Sigma_{ \{ \pm1, ..., \pm M \}}$ (on the other hand, showing $\Sigma_{S \cup \{ 0 \}} \neq \Sigma_{S \setminus \{ 0 \}}$ for a general set $S$ seems much harder). Clearly $1$ is rigid for normalized sets but an argument like in the early rigidity proof would not make sense for discussing removing $1$.

Using this observation we can make a more rigorous definition for `large' spikes. Let $S = \{0, \pm 1, \pm 2, ..., \pm M \}$ be a set of integers for some $M$. For each $1 < k < 2\sqrt{M}+1$ we define $\lambda_k$ be the root of $1 - (k+1)z + (M+k)z^2$ in the first quadrant (that is, the same root studied in \textbf{Theorem \ref{weakrigidity}}). Then we can define the $k^{th}$ spike to be the connected component of $\Sigma_S \setminus \Sigma_{S \setminus \{ \pm k \} }$ containing $\lambda_k$. Note that this definition of $k^{th}$ spike does not work for the $0^{th}$ and $1^{st}$ spike. Alternatively, we could define the $k^{th}$ spike as the intersection of $\Sigma_S$ with:

\[ 
\left \{ \lambda = \frac{1}{r}e^{i\theta} \in \mathbb{D} \mid \sqrt{M} \leqslant r \leqslant \sqrt{M}+1, \frac{-k - E_2(\lambda, S)}{2r} \leqslant \cos(\theta) \leqslant \frac{-k+E_2(\lambda, S)}{2r} \right \}.
\]

This definition works even for $k = 0$ and $1$ though it is larger than the previous characterization. Moreover, using this definition we would have that the spikes overlap.

\begin{figure}[!h]
\centering
\includegraphics[scale=0.7]{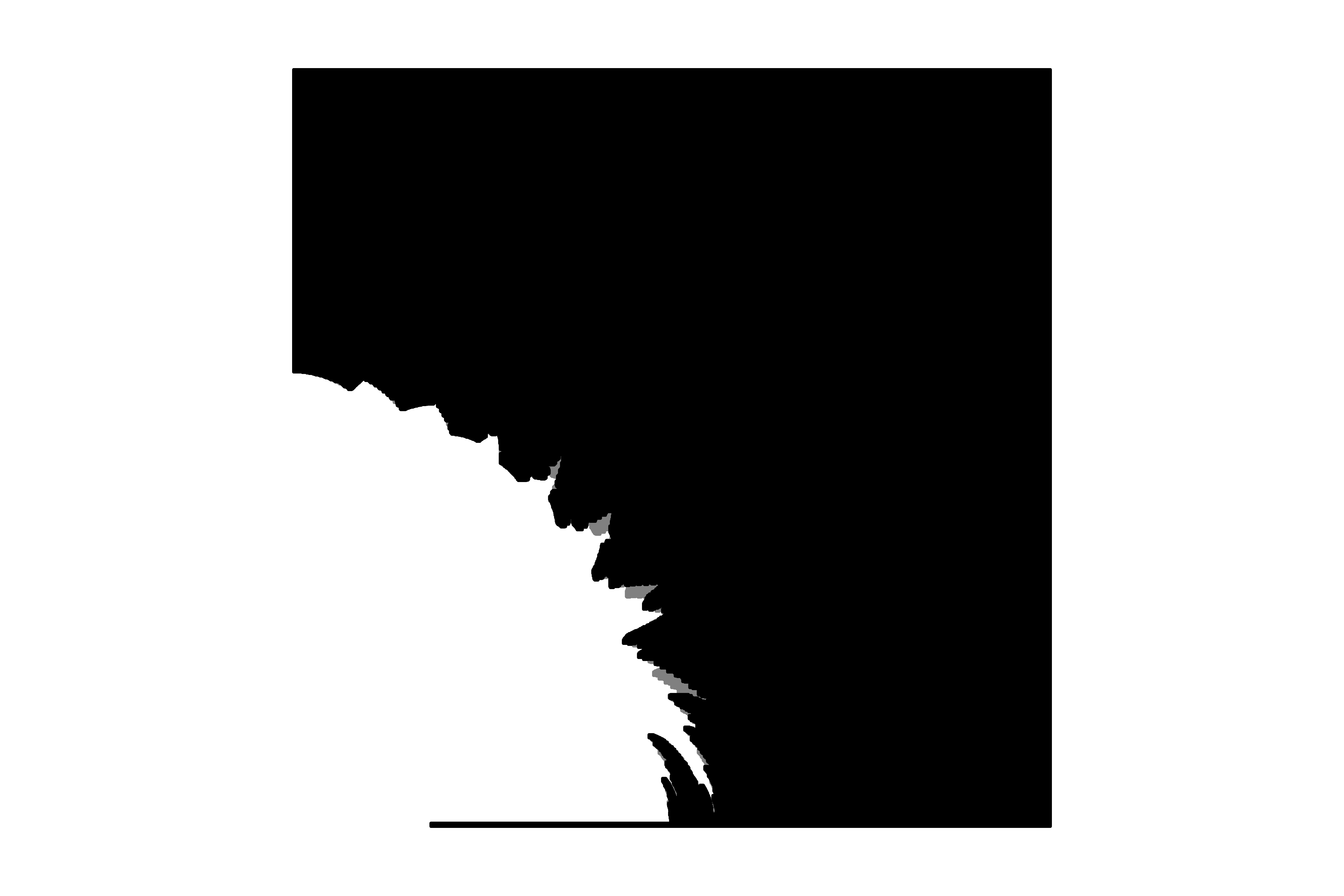}
\caption{In black we see the set of roots of power series with coefficients any integer from $-10$ to $10$ except $\pm 8$. In grey are the roots we get if we also allow $\pm 8$.}
\label{fig:0 to 10 no 8}
\end{figure}

\begin{figure}[!h]
\centering
\includegraphics[scale=0.7]{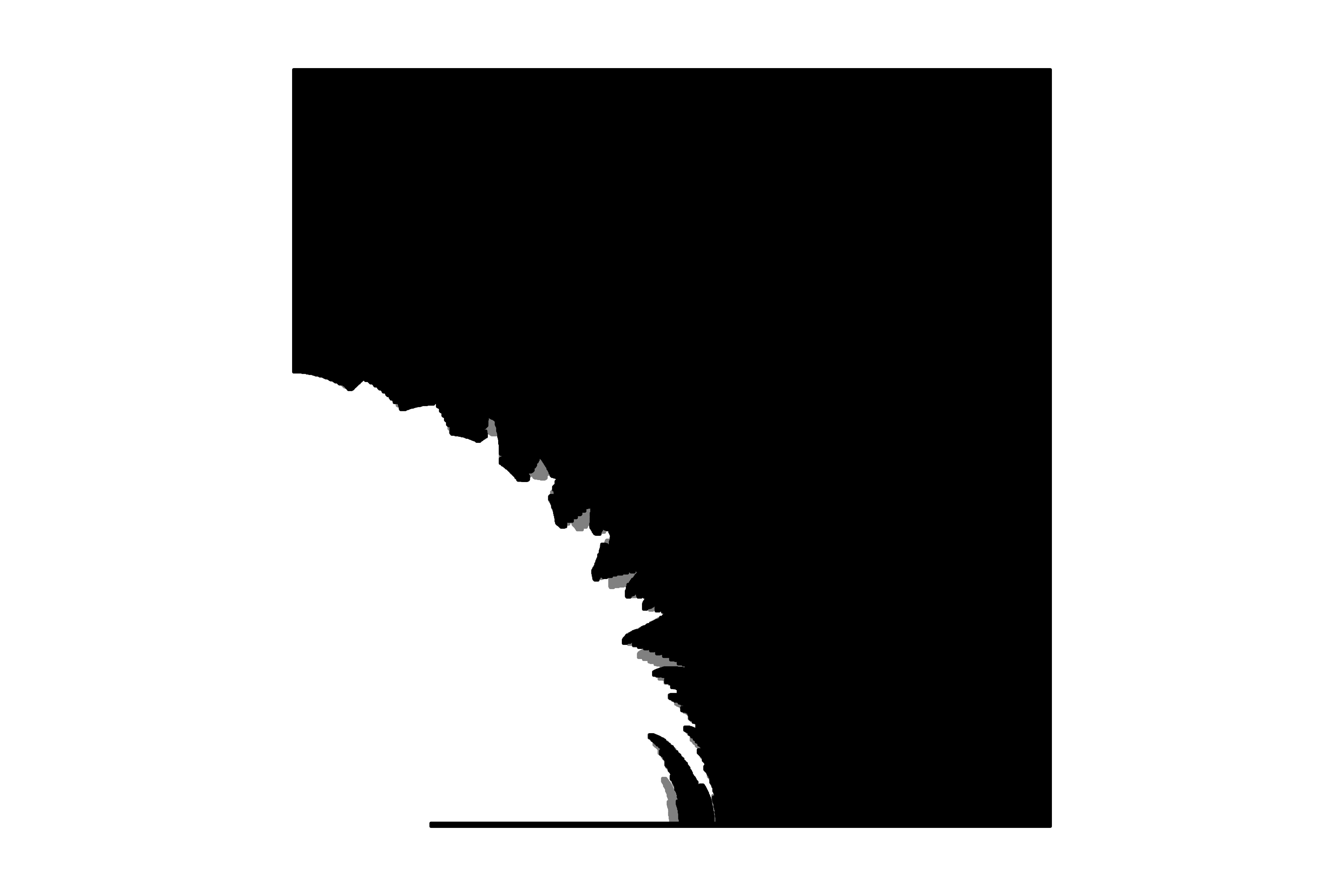}
\caption{In black we see the set of roots of power series with coefficients any integer from $-10$ to $10$ except $\pm 9$. In grey are the roots we get if we also allow $\pm 9$.}
\label{fig:0 to 10 no 9}
\end{figure}

Finally, let us consider the `larger' coefficients $\pm8$ and $\pm 9$. If we look at \textbf{Figures \ref{fig:0 to 10 no 8}} and \textbf{\ref{fig:0 to 10 no 9}}, we see that when we remove $\pm 8$ or $\pm 9$ no `large spike' gets removed, instead we see only minor differences in the finer details of a few spikes. In fact, the `large spikes' almost appear to be made up of smaller spikes which these `larger' coefficients effect. Unlike with the smaller coefficients we do not have any rigorous proof or description of this phenomenon. If, however, we assume this pattern holds for arbitrary sets $S$ then it would mean the small coefficients of $S$ essentially determine these `large spikes' while the larger coefficients determine minor detailing. By smaller we mean the coefficients $k$ satisfying $|k| < 2\sqrt{\max(S)} + 1$. A rigorous description of which coefficients are `large' seems more elusive however it would be of great interest if one could give an exact description of this phenomenon. In particular, it would be nice to know if any coefficient $k$ was always either `small' (in the sense above), or `large' (in the sense of effecting these `smaller spikes'). If the `small' and `large' coefficients do not overlap then it may be that the coefficients in between are not rigid and can be removed from the coefficient set without effecting the corresponding set of roots.

\end{document}